\documentclass[10pt,a4paper,article, twocolumn]{memoir} % memoir is nice
\usepackage[margin=0.8in]{geometry}

\counterwithout{section}{chapter}

\RequirePackage{hyperref}

\usepackage[american]{babel}
\usepackage{natbib} 
\usepackage{mathtools}
\usepackage{booktabs} 
\usepackage{tikz} 

\usepackage{amssymb}
\usepackage{amsthm}

\usepackage{nicefrac}

\usepackage{enumitem}

\newtheorem{lemma}{Lemma}
\newtheorem{proposition}{Proposition}
\newtheorem{corollary}{Corollary}
\newtheorem{theorem}{Theorem}

\theoremstyle{definition}

\newtheorem{assumption}{Assumption}

\newcommand{\reals}{\mathbb{R}}
\newcommand{\realsnonneg}{\mathbb{R}_{\geq0}}
\newcommand{\realspos}{\mathbb{R}_{>0}}
\newcommand{\nats}{\mathbb{N}}
\newcommand{\natswith}{\mathbb{N}_0}

\newcommand{\coloneqq}{:=}

\newcommand{\states}{\mathcal{X}}
\newcommand{\gamblesX}{\smash{\reals^\states}}
\newcommand{\gamblesY}[1]{\smash{\reals^{#1}}}
\newcommand{\gamblesAc}{\smash{\reals^{A^c}}}

\newcommand{\ones}{\mathbf{1}}

\newcommand{\abs}[1]{\left\lvert #1 \right\rvert}
\newcommand{\norm}[1]{\left\lVert #1 \right\rVert}

\newcommand{\resAc}{\vert_{A^c}}
\newcommand{\upX}{\!\!\uparrow_\mathcal{X}}

\newcommand{\lsubgen}{\underline{G}}
\newcommand{\usubgen}{\overline{G}}

\newcommand{\rateset}{\mathcal{Q}}
\newcommand{\lrate}{\underline{Q}}

\newcommand{\urate}{\overline{Q}}

\newcommand{\ind}[1]{\mathbb{I}_{#1}}

\newcommand{\rateforp}{{}^P\!Q}

\title{Hitting Times for Continuous-Time Imprecise-Markov Chains}

\author{Thomas Krak}
\date{\texttt{t.e.krak@tue.nl}\\Uncertainty in Artificial Intelligence -- Eindhoven University of Technology}

\begin{document}
\maketitle

\begin{abstract}
	We study the problem of characterizing the expected hitting times for a robust generalization of continuous-time Markov chains. This generalization is based on the theory of imprecise probabilities, and the models with which we work essentially constitute sets of stochastic processes. Their inferences are tight lower- and upper bounds with respect to variation within these sets. 
	
	We consider three distinct types of these models, corresponding to different levels of generality and structural independence assumptions on the constituent processes. 
	
	Our main results are twofold; first, we demonstrate that the hitting times for all three types are equivalent. Moreover, we show that these inferences are described by a straightforward generalization of a well-known linear system of equations that characterizes expected hitting times for traditional time-homogeneous continuous-time Markov chains.
\end{abstract}

\section{Introduction}\label{sec:intro}

We consider the problem of characterizing the \emph{expected hitting times} for continuous-time \emph{imprecise-Markov chains}~\citep{skulj2015efficient,krak2017imprecise,krak2021phd,erreygers2021phd}. These are \emph{robust}, set-valued generalizations of (traditional) Markov chains~\citep{norris1998markov}, based on the theory of \emph{imprecise probabilities}~\citep{Walley:1991vk,augustin2013:itip}. From a sensitivity-analysis perspective, we may interpret these sets as hedging against model-uncertainties with respect to a model's numerical parameters and/or structural (independence) assumptions.

The inference problem of hitting times essentially deals with the question of how long it will take the underlying system to reach some particular subset of its states. This is a common and important problem in such fields as, e.g., reliability analysis, where it can capture the expected time-to-failure of a system; and epidemiology, to model the expected time-until-extinction of an epidemic. For imprecise-Markov chains, then, we are interested in evaluating these quantities in a manner that is robust against, and conservative with respect to, any variation that is compatible with one's uncertainty about the model specification.

\citet{erreygers2021phd} has recently obtained some partial results towards characterizing such inferences, but has not been able to give a complete characterization and has largely studied the finite-time horizon case. The problem of hitting times for \emph{discrete-time} imprecise-Markov chains was previously studied by~\citet{krak2019hitting,krak2020computing}. In this present work, we largely emulate and extend their results to the continuous-time setting.

We will be concerned with three different types of imprecise-Markov chains. These are all sets of stochastic processes that are in a specific sense compatible with a given set of numerical parameters, but the three types differ in the independence properties of their elements. In particular, they correspond to (i) a set of (\emph{time-})\emph{homogeneous} Markov chains, (ii) a set of (not-necessarily homogeneous) Markov chains, and (iii) a set of general---not-necessarily homogeneous nor Markovian---stochastic processes. It is known (and perhaps not very surprising) that inferences with respect to these three models do not in general agree; see e.g.~\citep{krak2021phd} for a detailed analysis of their differences.

However, our first main result in this work is that the expected hitting time is \emph{the same} for these three different types of models. Besides being of theoretical interest, we want to emphasize the power of this result: it means that even if a practitioner using Markov chains would be uncertain whether the system they are studying is truly homogeneous and/or Markovian, relaxing these assumptions would not influence inferences about the hitting times in this sense. Purely pragmatically, it also means that we can use computational methods tailored to any one of these types of models, to compute these inferences. 

Our second main result is that these hitting times are characterized by a generalization of a well-known system of equations that holds for continuous-time homogeneous Markov chains; see Proposition~\ref{prop:precise_cont_system} for this linear system.

The remainder of this paper is structured as follows. In Section~\ref{sec:prelim} we introduce the basic required concepts that we will use throughout, formalizing the notion of stochastic processes and defining the inference problem of interest. In Section~\ref{sec:imp_markov}, we define the various types of imprecise-Markov chains that we use throughout this work. We spend some effort in Section~\ref{subsec:subspace_dynamics} to study the transition dynamics of these models, from a perspective that is particularly relevant for the inference problem of hitting times. In Section~\ref{sec:hits_as_limits} we explain and sketch the proofs of our main results, and we give a summary in Section~\ref{sec:summary}.

Because we have quite a lot of conceptual material to cover before we can explain our main results, we are not able to fit any real proofs in the main body of this work. Instead, these---together with a number of technical lemmas---have largely been relegated to the supplementary material.

\section{Preliminaries}\label{sec:prelim}

Throughout, we consider a fixed, finite \emph{state space} $\states$ with at least two elements. This set contains all possible values for some abstract underlying process. An element of $\mathcal{X}$ is called a \emph{state}, and is usually generically denoted as $x\in\mathcal{X}$.

We use $\reals, \realsnonneg$, and $\realspos$ to denote the reals, the non-negative reals, and the positive reals, respectively. $\nats$ denotes the natural numbers \emph{without} zero, and we let $\natswith\coloneqq\nats\cup\{0\}$. 

For any $\mathcal{Y}\subseteq\states$, we use $\gamblesY{\mathcal{Y}}$ to denote the vector space of real-valued functions on $\mathcal{Y}$; in particular, $\gamblesX$ denotes the space of all real functions on $\states$. We use $\norm{\cdot}$ to denote the supremum norm on any such space; for any $f\in\gamblesY{\mathcal{Y}}$ we let $\norm{f}\coloneqq \max\{\abs{f(x)}\,:\,x\in\mathcal{Y}\}$. Throughout, we make extensive use of \emph{indicator functions}, which are defined for all $A\subseteq \mathcal{Y}$ as $\ind{A}(x)\coloneqq 1$ if $x\in A$ and $\ind{A}(x)\coloneqq 0$, otherwise. We use the shorthand $\ind{y}\coloneqq\ind{\{y\}}$. Let $\ones$ denote the function that is identically equal to $1$; its dimensionality is to be understood from context. 

A map $M:\gamblesY{\mathcal{Y}}\to\gamblesY{\mathcal{Y}}$ is also called an \emph{operator}, and we denote its evaluation in $f\in\gamblesY{\mathcal{Y}}$ as $Mf$. If it holds for all $\lambda\in\realsnonneg$ that $M(\lambda f)=\lambda Mf$ then $M$ is called \emph{non-negatively homogeneous}. For any non-negatively homogeneous operator on $\gamblesY{\mathcal{Y}}$, we define the induced operator norm $\norm{M}\coloneqq \sup\{\norm{Mf}\,:\,f\in\gamblesY{\mathcal{Y}},\norm{f}=1\}$. We reserve the symbol $I$ to denote the identity operator on any space; the domain is to be understood from context.

Note that any \emph{linear} operator is also non-negatively homogeneous. Moreover, if $M$ is linear it can be represented as an $\abs{\mathcal{Y}}\times\abs{\mathcal{Y}}$ matrix by arbitrarily fixing an ordering on $\mathcal{Y}$. However, without fixing such an ordering, we simply use $M(x,y)\coloneqq M\ind{y}(x)$ to denote the entry in the $x$-row and $y$-column of such a matrix, for any $x,y\in\mathcal{Y}$. For any $f\in\gamblesY{\mathcal{Y}}$ and $x\in\mathcal{Y}$ we then have $Mf(x)=\sum_{y\in\mathcal{Y}}M(x,y)f(y)$, so that $Mf$ simply represents the usual matrix-vector product of $M$ with the (column) vector $f$. In the sequel, we interchangeably refer to linear operators also as matrices. We note the well-known equality $\norm{M}=\max_{x\in\mathcal{Y}}\sum_{y\in\mathcal{Y}}\abs{M(x,y)}$ for the induced matrix norm.

\subsection{Processes \& Markov Chains}

We now turn to stochastic processes, which are fundamentally the subject of this work. The typical (measure-theoretic) way to define a stochastic process is simply as a family $(X_i)_{i\in\mathcal{I}}$ of random variables with index set $\mathcal{I}$. This index set represents the time domain of the stochastic process. The random variables are understood to be taken with respect to some underlying probability space $(\Omega_{\mathcal{I}},\mathcal{F}_{\mathcal{I}},P)$, where $\Omega_\mathcal{I}$ is a set of \emph{sample paths}, which are functions from $\mathcal{I}$ to $\states$ representing possible realizations of the evolution of the underlying process through $\states$. The random variables $X_i$, $i\in\mathcal{I}$ are canonically the maps $X_i:\omega\mapsto \omega(i)$ on $\Omega_{\mathcal{I}}$. 

However, for our purposes it will be more convenient to instead refer to the \emph{probability measure} $P$ as the stochastic process. Different processes $P$ may then be taken over the same measurable space $(\Omega_\mathcal{I},\mathcal{F}_\mathcal{I})$, using the same canonical variables $(X_i)_{i\in\mathcal{I}}$ for all these processes.

In this work we will use both \emph{discrete}- and \emph{continuous}-time stochastic processes, which corresponds to choosing $\mathcal{I}=\natswith$ or $\mathcal{I}=\realsnonneg$, respectively. In both cases we take $\mathcal{F}_\mathcal{I}$ to be the $\sigma$-algebra generated by the cylinder sets; this ensures that all functions that we consider are measurable. 

In the discrete-time case, we let $\Omega_{\natswith}$ be the set of \emph{all} functions from $\natswith$ to $\states$. A discrete-time stochastic process $P$ is then simply a probability measure on $(\Omega_{\natswith},\mathcal{F}_\mathbf{\natswith})$. Moreover, $P$ is said to be a \emph{Markov chain} if it satisfies the (discrete-time) \emph{Markov property}, meaning that
\begin{align*}
	P(X_{n+1}=x_{n+1}\,&\vert\,X_0=x_0,\ldots,X_{n}=x_n) \\
	&= P(X_{n+1}=x_{n+1}\,\vert\,X_n=x_{n})\,,
\end{align*}
for all $x_0,\ldots,x_{n+1}\in\states$ and $n\in\natswith$. If, additionally, it holds for all $x,y\in\states$ and $n\in\natswith$ that
\begin{equation*}
	P(X_{n+1}=y\,\vert\,X_n=x)=P(X_1=y\,\vert\,X_0=x)\,,
\end{equation*}
then $P$ is said to be a \mbox{(\emph{time-})\emph{homogeneous}} Markov chain. We use $\mathbb{P}_{\natswith}, \mathbb{P}_{\natswith}^{\mathrm{M}}$, and $\mathbb{P}_{\natswith}^{\mathrm{HM}}$ to denote, respectively, the set of \emph{all} discrete-time stochastic processes; the set of all discrete-time Markov chains; and the set of all discrete-time homogeneous Markov chains.

In the continuous-time case, we let $\Omega_{\realsnonneg}$ be the set of all \emph{cadlag} functions from $\realsnonneg$ to $\states$. A continuous-time stochastic process $P$ is a probability measure on $(\Omega_{\realsnonneg},\mathcal{F}_{\realsnonneg})$. The process $P$ is said to be a Markov chain if it satisfies the (continuous-time) Markov property,
\begin{align*}
	P(X_{t_{n+1}}=x_{t_{n+1}}\,&\vert\,X_{t_0}=x_{t_0},\ldots,X_{t_n}=x_{t_n}) \\
	&= P(X_{t_{n+1}}=x_{t_{n+1}}\,\vert\,X_{t_n}=x_{t_n})
\end{align*}
for all $x_{t_0},\ldots,x_{t_{n+1}}\in\states$, $t_0<\cdots< t_n\leq t_{n+1}\in\realsnonneg$, and all $n\in\natswith$. If, additionally, it holds that
\begin{equation*}
	P(X_s=y\,\vert\,X_t=x) = P(X_{s-t}=y\,\vert\,X_0=x)
\end{equation*}
for all $x,y\in\states$ and all $t,s\in\realsnonneg$ with $t\leq s$, then $P$ is said to be a \mbox{(time-)homogeneous} Markov chain. 
We use $\mathbb{P}_{\realsnonneg}, \mathbb{P}_{\realsnonneg}^{\mathrm{M}}$, and $\mathbb{P}_{\realsnonneg}^{\mathrm{HM}}$ to denote, respectively, the set of \emph{all} continuous-time stochastic processes; the set of all continuous-time Markov chains; and the set of all continuous-time homogeneous Markov chains.

We refer to~\citep{norris1998markov} for an excellent further introduction to discrete-time and continuous-time Markov chains.

\subsection{Transition Dynamics}\label{subsec:precise_dynamics}

Throughout this work, we make extensive use of operator-theoretic representations of the behavior of stochastic processes, and Markov chains in particular. The first reason for this is that such operators serve as a way to parameterize Markov chains. Moreover, they are also useful as a \emph{computational} tool, since they can often be used to express inferences of interest; see, e.g., Propositions~\ref{prop:precise_discr_system} and~\ref{prop:precise_cont_system} further on. We introduce the basic concepts below, and refer to e.g.~\citep{norris1998markov} for details.

A \emph{transition matrix} $T$ is a linear operator on $\gamblesX$ such that, for all $x\in\states$, it holds that $T(x,y)\geq 0$ for all $y\in\states$, and $\sum_{y\in\states}T(x,y)=1$. There is an important and well-known connection between Markov chains and transition matrices; for any discrete-time homogeneous Markov chain $P$, we can define the \emph{corresponding transition matrix} ${}^{P}T$ as
\begin{equation*}
	{}^{P}T(x,y)\coloneqq P(X_{1}=y\,\vert\,X_0=x)\quad\text{for all $x,y\in\states$.}
\end{equation*}
Since $P$ is a probability measure, we clearly have that ${}^{P}T$ is a transition matrix. Conversely, a given transition matrix $T$ uniquely determines a discrete-time homogeneous Markov chain $P$ with ${}^{P}T=T$, up to the specification of the initial distribution $P(X_0)$. For this reason, transition matrices are often taken as a crucial parameter to specify (discrete-time, homogeneous) Markov chains.

Analogously, for a (non-homogeneous) discrete-time Markov chain $P$, we might define a family $({}^PT_n)_{n\in\natswith}$ of \emph{time-dependent} corresponding transition matrices, with
\begin{equation*}
	{}^{P}T_n(x,y)\coloneqq P(X_{n+1}=y\,\vert\,X_n=x)\,,
\end{equation*}
for all $x,y\in\states$ and $n\in\natswith$. Conversely, any family $(T_n)_{n\in\natswith}$ of transition matrices uniquely determines a discrete-time Markov chain $P$ with ${}^{P}T_n=T_n$ for all $n\in\natswith$, again up to the specification of $P(X_0)$.

In the continuous-time setting, transition matrices are also of great importance. However, it will be instructive to first introduce rate matrices.
A \emph{rate matrix} $Q$ is a linear operator on $\gamblesX$ such that, for all $x\in\states$, it holds that $Q(x,y)\geq 0$ for all $y\in\states$ with $x\neq y$, and $\sum_{y\in\states}Q(x,y)=0$.

For any rate matrix $Q$ and any $t\in\realsnonneg$, the \emph{matrix exponential} $e^{Qt}$ of $Qt$ can be defined as~\citep{van2006study}
\begin{equation*}
	e^{Qt}\coloneqq \lim_{n\to+\infty}\bigl(I+\nicefrac{t}{n}Q\bigr)^n\,.
\end{equation*}
An alternative characterization is as the (unique) solution to the matrix ordinary differential equation~\citep{van2006study}
\begin{equation}\label{eq:matrix_exp_differential}
	\frac{\mathrm{d}}{\mathrm{d}\,s}e^{Qs} = Qe^{Qs}=e^{Qs}Q,\quad\text{with $e^{Q0}=I$.}
\end{equation}
For any $t,s\in\realsnonneg$ it holds that $e^{Q(t+s)}=e^{Qt}e^{Qs}$, and we immediately have $e^{Q0}=I$. The family $(e^{Qt})_{t\in\realsnonneg}$ is therefore called the \emph{semigroup} generated by~$Q$, and $Q$ is called the \emph{generator} of this semigroup. Moreover, for any rate matrix $Q$ and any $t\in\realsnonneg$, $e^{Qt}$ is a transition matrix~\citep[Thm 2.1.2]{norris1998markov}.

Now let us consider a continuous-time homogeneous Markov chain $P$, and define the corresponding transition matrix\footnote{Note that in continuous-time, we always have to measure the transition-time interval $[0,t]$ to specify these matrices.} ${}^PT_t$ for all $t\in\realsnonneg$ and $x,y\in\states$ as
\begin{equation}\label{eq:trans_mat_continuous}
	{}^PT_t(x,y) \coloneqq P(X_t = y\,\vert\,X_0=x)\,.
\end{equation}
It turns out that there is then a unique rate matrix $\rateforp$ associated with $P$ such that ${}^PT_t=e^{\rateforp t}$ for all $t\in\realsnonneg$. By combining Equations~\eqref{eq:matrix_exp_differential} and~\eqref{eq:trans_mat_continuous}, we can identify $\rateforp$ as
\begin{equation*}
	\rateforp = \Bigl(\frac{\mathrm{d}}{\mathrm{d}\,t}{}^PT_t\Bigr) \bigg\vert_{t=0}\,.
\end{equation*}
As before, in the other direction we have that any fixed rate matrix $Q$ uniquely determines a continuous-time homogeneous Markov chain $P$ with $\rateforp=Q$, up to the specification of $P(X_0)$. For this reason, rate matrices are often used to specify (continuous-time, homogeneous) Markov chains.

Let us finally consider a (not-necessarily homogeneous) continuous-time Markov chain $P$. For any $t,s\in\realsnonneg$ with $t\leq s$, we can then define a transition matrix ${}^PT_t^s$ with, for all $x,y\in\states$, ${}^PT_t^s(x,y) \coloneqq P(X_s=y\,\vert\,X_t=x)$.
Under appropriate assumptions of differentiability, this induces a family $(\rateforp_t)_{t\in\realsnonneg}$ of rate matrices $\rateforp_t$, as
\begin{equation}\label{eq:time_dependent_rate}
	\rateforp_t = \Bigl(\frac{\mathrm{d}}{\mathrm{d}\,s}{}^PT_t^s\Bigr) \bigg\vert_{s=t}\,.
\end{equation}
In the converse direction we might try to reconstruct the transition matrices of $P$ by solving the matrix ordinary differential equation(s)
\begin{equation}\label{eq:nonhomogen_diffential_form}
	\frac{\mathrm{d}}{\mathrm{d}\,s}{}^PT_t^s = {}^PT_t^s\rateforp_s,\quad\text{with ${}^PT_t^t=I$.}
\end{equation}
By comparing with Equation~\eqref{eq:matrix_exp_differential}, we see that in the special case where $\rateforp_s$ does not depend on $s$---that is, where $P$ is homogeneous with $\rateforp_s=\rateforp$, say---we indeed obtain ${}^PT_t^s=e^{\rateforp(s-t)}$. However, in general the \emph{non-autonomous} system~\eqref{eq:nonhomogen_diffential_form} does not have such a closed-form solution, and we cannot move beyond this implicit characterization.

\subsection{Hitting Times}

We now have all the pieces to introduce the inference problem that is the subject of this work, \emph{viz}. the \emph{expected hitting times} of some non-empty set of states $A\subset\states$ with respect to a particular stochastic process. We take this set $A$ to be fixed for the remainder of this work.

In the discrete-time case, we consider the (extended real-valued)\footnote{We agree that $0(+\infty)=0$; $(+\infty) + (+\infty) = +\infty$; and, for any $c\in\reals$, $(+\infty)+c=+\infty$ and $c(+\infty)=+\infty$ if $c>0$.} function $\tau_{\natswith}:\Omega_{\natswith}\to\realsnonneg\cup\{+\infty\}$ given by
\begin{equation*}
	\tau_{\natswith}(\omega) \coloneqq \inf\bigl\{n\in\natswith\,:\,\omega(n)\in A\bigr\}\quad\text{for all $\omega\in\Omega_{\natswith}$.}
\end{equation*}
This captures the number of steps before a process $P$ ``hits'' any state in $A$.
The expected hitting time for a discrete-time process $P$ starting in $x\in\states$ is then defined as
\begin{equation*}
	\mathbb{E}_P\bigl[\tau_{\natswith}\,\vert\,X_0=x\bigr] \coloneqq \int_{\Omega_{\natswith}} \tau_{\natswith}(\omega)\,\mathrm{d}P(\omega\,\vert\,X_0=x)\,.
\end{equation*}
We use $\mathbb{E}_P\bigl[\tau_{\natswith}\,\vert\,X_0\bigr]$ to denote the extended real-valued function on $\states$ given by $x\mapsto \mathbb{E}_P\bigl[\tau_{\natswith}\,\vert\,X_0=x\bigr]$. When dealing with homogeneous Markov chains, this quantity has the following simple characterization:
\begin{proposition}{\citep[Thm 1.3.5]{norris1998markov}}\label{prop:precise_discr_system}
	Let $P$ be a discrete-time homogeneous Markov chain with corresponding transition matrix ${}^{P}T$. Then $h\coloneqq \mathbb{E}_P\bigl[\tau_{\natswith}\,\vert\,X_0\bigr]$ is the minimal non-negative solution to the linear system\footnote{Throughout, for any $f,g\in\gamblesX$, the quantity $fg$ is understood as the pointwise product between the functions $f$ and $g$.}\footnote{Strictly speaking this requires extending the domain of ${}^PT$ to extended-real valued functions, but we will shortly introduce some assumptions that obviate such an exposition.}
	\begin{equation*}
		h = \ind{A^c} + \ind{A^c} {}^{P}Th\,.
	\end{equation*}
\end{proposition}
In the continuous-time case, the definition is analogous; we introduce a function $\tau_{\realsnonneg}:\Omega_{\realsnonneg}\to\realsnonneg\cup\{+\infty\}$ as
\begin{equation*}
	\tau_{\realsnonneg}(\omega) \coloneqq \inf\bigl\{t\in\realsnonneg\,:\,\omega(t)\in A\bigr\}\,\,\text{for all $\omega\in\Omega_{\realsnonneg}$.}
\end{equation*}
This function measures the time until a process ``hits'' any state in $A$ on a given sample path. The expected hitting time for a continuous-time process $P$ starting in $x\in\states$ is
\begin{equation*}
	\mathbb{E}_P\bigl[\tau_{\realsnonneg}\,\vert\,X_0=x\bigr] \coloneqq \int_{\Omega_{\realsnonneg}} \tau_{\realsnonneg}(\omega)\,\mathrm{d}P(\omega\,\vert\,X_0=x)\,.
\end{equation*}
We again use $\smash{\mathbb{E}_P\bigl[\tau_{\realsnonneg}\,\vert\,X_0\bigr]}$ to denote the extended-real valued function on $\states$ given by $x\mapsto \smash{\mathbb{E}_P\bigl[\tau_{\realsnonneg}\,\vert\,X_0=x\bigr]}$. Also in this case, the characterization for homogeneous Markov chains is particularly simple:
\begin{proposition}{\citep[Thm 3.3.3]{norris1998markov}}\label{prop:precise_cont_system}
	Let $P$ be a continuous-time homogeneous Markov chain with rate matrix $\smash{\rateforp}$ such that $\smash{\rateforp}(x,x)\neq 0$ for all $x\in A^c$. Then $h\coloneqq \smash{\mathbb{E}_P\bigl[\tau_{\realsnonneg}\,\vert\,X_0\bigr]}$ is the minimal non-negative solution to
	\begin{equation}\label{eq:prop:precise_cont_system}
		\ind{A}h = \ind{A^c} + \ind{A^c}\smash{\rateforp} h\,.
	\end{equation}
\end{proposition}

\section{Imprecise-Markov Chains}\label{sec:imp_markov}

Let us now introduce \emph{imprecise-Markov chains}~\citep{itip:stochasticprocesses,skulj2015efficient,krak2017imprecise}, which are the stochastic processes that we aim to study in this work. Their characterization is based on the theory of \emph{imprecise probabilities}~\citep{Walley:1991vk,augustin2013:itip}. 

We here adopt the ``sensitivity analysis'' interpretation of imprecise probabilities. This means that we represent an imprecise-Markov chain simply as a \emph{set} $\mathcal{P}$ of stochastic processes. Intuitively, the idea is that we collect in $\mathcal{P}$ all (traditional, ``precise'') stochastic processes that we deem to plausibly capture the dynamics of the underlying system of interest. Inferences with respect to $\mathcal{P}$ are defined using \emph{lower-} and \emph{upper} expectations, given respectively as
\begin{equation*}
	\underline{\mathbb{E}}_\mathcal{P}[\cdot\,\vert\,\cdot]\coloneqq \inf_{P\in\mathcal{P}}\mathbb{E}_P[\cdot\,\vert\,\cdot]
	\quad\text{and}\quad
	\overline{\mathbb{E}}_\mathcal{P}[\cdot\,\vert\,\cdot]\coloneqq \sup_{P\in\mathcal{P}}\mathbb{E}_P[\cdot\,\vert\,\cdot]\,.
\end{equation*}
So, their inferences represent \emph{robust}---i.e. conservative---and \emph{tight} lower- and upper bounds on inferences with respect to \emph{all} stochastic processes that we deem to be plausible.

\subsection{Sets of Processes \& Types}\label{subsec:imc_sets_types}

We already mentioned that an imprecise-Markov chain is essentially simply a set $\mathcal{P}$ of stochastic processes. Let us now consider how to define such sets.

We start by considering the discrete-time case; then, clearly, $\mathcal{P}$ will be a set of discrete-time processes. We will parameterize such a set with some non-empty set $\mathcal{T}$ of transition matrices. Our aim is then to include in $\mathcal{P}$ all processes that are in some sense ``compatible'' with $\mathcal{T}$.\footnote{We will not constrain the initial models $P(X_0)$ of the elements of $\mathcal{P}$, since in any case such a choice would not influence the inferences that we study in this work.} However, at this point we are faced with a choice about which \emph{type} of processes to include in this set, and these different choices lead to \emph{different types of imprecise-Markov chains}.

Arguably the conceptually most simple model is $\mathcal{P}^{\mathrm{HM}}_\mathcal{T}$, which contains all homogeneous Markov chains $P$ whose corresponding transition matrix is included in $\mathcal{T}$:
\begin{equation*}
	\mathcal{P}^{\mathrm{HM}}_\mathcal{T}\coloneqq \bigl\{ P\in\mathbb{P}^{\mathrm{HM}}_{\natswith}\,:\, {}^{P}T\in\mathcal{T} \bigr\}\,.
\end{equation*} 
However, we could instead consider $\mathcal{P}^\mathrm{M}_{\mathcal{T}}$, which is the set of all (not-necessarily homogeneous) Markov chains whose time-dependent transition matrices are contained in $\mathcal{T}$:
\begin{equation*}
	\mathcal{P}^{\mathrm{M}}_\mathcal{T}\coloneqq \bigl\{ P\in\mathbb{P}^{\mathrm{M}}_{\natswith}\,:\, {}^{P}T_n\in\mathcal{T}\,\text{for all $n\in\natswith$} \bigr\}\,.
\end{equation*} 
The last choice that we consider here is the set $\mathcal{P}^{\mathrm{I}}_{\mathcal{T}}$, which essentially contains \emph{all} discrete-time processes whose single-step transition dynamics are described by $\mathcal{T}$. Its characterization is more cumbersome since we have not expressed these general processes in terms of transition matrices, but we can say that it is the set of all $P\in\mathbb{P}_{\natswith}$ such that for all $n\in\natswith$ and all $x_0,\ldots,x_{n}\in\states$, there is some $T\in\mathcal{T}$ such that for all $y\in\states$ it holds that
\begin{equation*}
	P(X_{n+1}=y\,\vert\,X_0=x_0,\ldots,X_n=x_n) = T(x_n,y)\,.
\end{equation*}
This last type is called an imprecise-Markov chain under \emph{epistemic irrelevance}, whence the superscript `$\mathrm{I}$'.

Note that the three types $\mathcal{P}^\mathrm{HM}_{\mathcal{T}}, \mathcal{P}^\mathrm{M}_{\mathcal{T}}$, and $\mathcal{P}^\mathrm{I}_{\mathcal{T}}$ capture not only ``plausible'' variation in terms of parameter uncertainty---expressed through the set $\mathcal{T}$---but also variation in terms of the structural independence conditions that we consider! So, from an applied perspective, if someone is not sure whether the underlying system that they are studying is truly Markovian and/or time-homogeneous, they might choose to use different such sets in their analysis.

In the continuous-time case, we again proceed analogously. First, we fix a non-empty set $\rateset$ of rate matrices, which will be the parameter for our models. We then first consider the set $\mathcal{P}^{\mathrm{HM}}_\rateset$ of all homogeneous Markov chains whose rate matrix is included in $\rateset$:
\begin{equation*}
	\mathcal{P}^{\mathrm{HM}}_\rateset \coloneqq \bigl\{ P\in\mathbb{P}_{\realsnonneg}^{\mathrm{HM}}\,:\, \rateforp\in\rateset\bigr\}\,.
\end{equation*}
The other two types are constructed in analogy to the discrete-time case, but unfortunately we don't have the space for a complete exposition of their characterization. Instead we refer the interested reader to~\citep{krak2017imprecise,krak2021phd} for an in-depth study of these different types and comparisons between them; in what follows we limit ourselves to a largely intuitive specification. 

The model $\mathcal{P}^\mathrm{M}_\rateset$ is the set of all continuous-time (not-necessarily homogeneous) Markov chains whose transition dynamics are compatible with $\rateset$ at every point in time. This includes in particular all Markov chains $P$ satisfying the appropriate differentiability assumptions to meaningfully say that the time-dependent rate matrices $\rateforp_t$---as in Equation~\eqref{eq:time_dependent_rate}---are included in $\rateset$ for all $t\in\realsnonneg$. However, $\mathcal{P}^\mathrm{M}_\rateset$ also contains other processes that are not (everywhere) differentiable; see e.g.~\citep[Sec 4.6 and 5.2]{krak2021phd} for the technical details.

The most involved model to explain is again $\mathcal{P}^\mathrm{I}_\rateset$, which includes \emph{all} continuous-time processes whose time- and history-dependent transition dynamics can be described using elements of $\rateset$. It includes, but is not limited to, appropriately differentiable processes $P$ such that for all $n\in\natswith$, all $t_0<\cdots<t_n\in\realsnonneg$, and all $x_{t_0},\ldots,x_{t_{n}}\in\states$, there is some $Q\in\rateset$ such that for all $y\in\states$ it holds that
\begin{align*}
	\biggl(\frac{\mathrm{d}}{\mathrm{d}\,s}P(X_s=y\,\vert\,X_{t_0}=x_{t_0},\ldots,&X_{t_n}=x_{t_n})\biggr)\bigg\vert_{s=t_n} \\
	&= Q(x_{t_n},y)
\end{align*}
We again refer to~\citep[Sec 4.6 and 5.2]{krak2021phd} for the technical details involving the additional elements of $\mathcal{P}^\mathrm{I}_\rateset$ that are not appropriately differentiable. Importantly, we note the nested structure~~\citep[Prop 5.9]{krak2021phd}
\begin{equation*}
	\mathcal{P}^\mathrm{HM}_\rateset \subseteq \mathcal{P}^\mathrm{M}_\rateset \subseteq \mathcal{P}^\mathrm{I}_\rateset\,,
\end{equation*}
where the inclusions are strict provided $\rateset$ isn't trivial.

For notational convenience, we will use identical sub- and superscripts to denote the corresponding lower- and upper expectations for any of these imprecise-Markov chains; e.g., we let $\underline{\mathbb{E}}^\mathrm{HM}_\mathcal{T}[\cdot\,\vert\,\cdot] \coloneqq \underline{\mathbb{E}}_{\mathcal{P}^\mathrm{HM}_\mathcal{T}}[\cdot\,\vert\,\cdot]$.

\subsection{Imprecise Transition Dynamics}\label{subsec:imprecise_dynamics}

Let us now introduce some machinery to describe the dynamics of imprecise-Markov chains. In particular, we here move from the set-valued parameters $\mathcal{T}$ and $\rateset$ used in Section~\ref{subsec:imc_sets_types}, to their dual representations; these are operators that can serve as computational tools. 

In Section~\ref{subsec:imc_sets_types}, we described discrete-time imprecise-Markov chains using non-empty sets $\mathcal{T}$ of transition matrices. With any such set, we can associate the corresponding \emph{lower-} and \emph{upper transition operators} $\underline{T}$ and $\overline{T}$ on $\gamblesX$, defined respectively as
\begin{equation*}
	\underline{T}f\coloneqq \inf_{T\in\mathcal{T}}Tf
	\quad\text{and}\quad
	\overline{T}f\coloneqq \sup_{T\in\mathcal{T}}Tf
	\quad\text{for all $f\in\gamblesX$.}
\end{equation*}
More generally, any operator $\underline{T}$ (resp. $\overline{T}$) on $\gamblesX$ is a \emph{lower} (resp. \emph{upper}) \emph{transition operator} if for all $f,g\in\gamblesX$, all $\lambda\in\realsnonneg$, and all $x\in\states$, it holds that~\citep{de2017limit}
\begin{enumerate}
	\item $\min_{y\in\states}f(y)\leq\underline{T}f(x)$ and $\overline{T}f(x)\leq \max_{y\in\states}f(y)$
	\item $\underline{T}f + \underline{T}g \leq \underline{T}(f+g)$ and $\overline{T}(f+g)\leq \overline{T}f+\overline{T}g$
	\item $\underline{T}(\lambda f)=\lambda\underline{T}f$ and $\overline{T}(\lambda f)=\lambda\overline{T}f$.
\end{enumerate} 
It should be noted that lower- and upper transition operators are conjugate, in that any $\underline{T}$ induces a corresponding upper transition operator $\overline{T}(\cdot)=-\underline{T}(-\cdot)$, and \emph{vice versa}. Moreover, any transition matrix $T$ is also a lower---and, by its linearity, upper---transition operator.

It is easily verified that the lower- and upper transition operators corresponding to a given non-empty set $\mathcal{T}$ are, indeed, lower- and upper transition operators. Conversely, with a given lower transition operator $\underline{T}$, we can associate the set of transition matrices that \emph{dominate} it, in the sense that
\begin{equation*}
	\mathcal{T}_{\underline{T}}\coloneqq \bigl\{ T\,:\, \text{$T$ a trans. mat.},\,Tf\geq \underline{T}f\,\text{for all $f\in\gamblesX$}\bigr\}\,.
\end{equation*}
This set satisfies the following important properties:
\begin{proposition}{\citep[Sec 3.4]{krak2021phd}}\label{prop:duality_trans}
	Let $\underline{T}$ be a lower transition operator with conjugate upper transition operator $\overline{T}(\cdot)=-\underline{T}(-\cdot)$ and dominating set of transition matrices $\mathcal{T}_{\underline{T}}$. Then $\mathcal{T}_{\underline{T}}$ is a non-empty, closed, and convex set of transition matrices that has separately specified rows,\footnote{A set $\mathcal{M}$ of matrices is said to have \emph{separately specified rows} if, intuitively, it is closed under the row-wise recombination of its elements; see e.g.~\citep{itip:stochasticprocesses} for details.} and for all $f\in\gamblesX$ it holds that $\underline{T}f = \inf_{T\in\mathcal{T}_{\underline{T}}} Tf$ and $\overline{T}f = \sup_{T\in\mathcal{T}_{\underline{T}}} Tf$.
	%\begin{equation*}
	%	\underline{T}f = \inf_{T\in\mathcal{T}_{\underline{T}}} Tf
	%	\quad\text{and}\quad
	%	\overline{T}f = \sup_{T\in\mathcal{T}_{\underline{T}}} Tf\,.
	%\end{equation*}
	Moreover, for all $f\in\gamblesX$ there is some $T\in\mathcal{T}_{\underline{T}}$ such that $Tf=\underline{T}f$, and there is some---possibly different---$T\in\mathcal{T}_{\underline{T}}$ such that $Tf=\overline{T}f$.	
\end{proposition}
Notably, there is a one-to-one relation between non-empty sets of transition matrices that are closed and convex and have separately specified rows, and lower (or upper) transition operators: if $\underline{T}$ is the lower transition operator for the set $\mathcal{T}$, and if $\mathcal{T}$ satisfies these properties, then $\mathcal{T}=\mathcal{T}_{\underline{T}}$~\citep[Cor 3.38]{krak2021phd}. Hence these objects may serve as dual representations for each other. 

One reason that this is important is the use of $\underline{T}$ as a computational tool; under the conditions of this duality it holds that for any function $f\in\gamblesX$ and any $n\in\natswith$, we can write~\citep{itip:stochasticprocesses}
\begin{equation*}
	\underline{\mathbb{E}}^\mathrm{I}_{\mathcal{T}}[f(X_n)\vert X_0=x] = \underline{\mathbb{E}}^\mathrm{M}_{\mathcal{T}}[f(X_n)\vert X_0=x] = \underline{T}^nf(x)\,,
\end{equation*}
where $\underline{T}$ is the lower transition operator for $\mathcal{T}$. This reduces the problem of computing such inferences for the imprecise-Markov chains $\mathcal{P}^\mathrm{M}_\mathcal{T}$ and $\mathcal{P}^\mathrm{I}_\mathcal{T}$ to solving $n$ independent \emph{linear} optimization problems over $\mathcal{T}$; first compute $f_1\coloneqq\underline{T}f$, then compute $f_2\coloneqq\underline{T}\,f_1=\underline{T}^2f$, and so forth. Note that this method in general only yields a conservative bound on the corresponding inference for $\mathcal{P}^\mathrm{HM}_\mathcal{T}$, as the minimizers $T_k$ that obtain $T_kf_{k-1}=\underline{T} f_{k-1}$ may be different at each step.

We next consider the dynamics in the continuous-time setting. We proceed analogously to the above: we first consider a non-empty and bounded\footnote{In the induced operator norm.} set $\rateset$ of rate matrices. With this set, we then associate the corresponding \emph{lower-} and \emph{upper rate operators} $\lrate$ and $\urate$ on $\gamblesX$, defined as
\begin{equation*}
	\lrate f\coloneqq \inf_{Q\in\rateset}Qf
	\quad\text{and}\quad
	\urate f\coloneqq \sup_{Q\in\rateset}Qf
	\quad\text{for all $f\in\gamblesX$.}
\end{equation*}
More generally, any operator $\lrate$ (resp. $\urate$) on $\gamblesX$ is a \emph{lower} (resp. \emph{upper}) \emph{rate operator} if for all $f,g\in\gamblesX$, all $\lambda\in\realsnonneg$ and $\mu\in\reals$, and all $x,y\in\states$ with $y\neq x$, it holds that~\citep{de2017limit}
\begin{enumerate}
	\item $\lrate(\mu\ones)(x)=0$ and $\urate(\mu\ones)(x)=0$
	\item $\lrate\ind{y}(x)\geq 0$ and $\urate\ind{y}(x)\geq 0$
	\item $\lrate f + \lrate g \leq \lrate(f+g)$ and $\urate(f+g)\leq \urate f+\urate g$
	\item $\lrate(\lambda f)=\lambda\lrate f$ and $\urate(\lambda f)=\lambda\urate f$
\end{enumerate}
As before, such objects are conjugate, in that if $\lrate$ is a lower rate operator, then $\smash{\urate}(\cdot)=-\smash{\lrate}(-\cdot)$ is an upper rate operator. Moreover, any rate matrix $Q$ is also a lower (and upper) rate operator.
There is again a duality between lower (or upper) rate operators, and sets of rate matrices. For fixed $\lrate$ and with the dominating set of rate matrices $\rateset_{\lrate}$ defined as
\begin{equation*}
	\rateset_{\lrate}\coloneqq \bigl\{ Q\,:\, \text{$Q$ a rate mat.},\,Qf\geq \lrate f\,\text{for all $f\in\gamblesX$}\bigr\}\,,
\end{equation*}
we have the following result:
\begin{proposition}{\citep[Sec 6.2]{krak2021phd}}\label{prop:duality_rate}
	Let $\underline{Q}$ be a lower rate operator with conjugate upper rate operator $\overline{Q}(\cdot)=-\underline{Q}(-\cdot)$ and dominating set of rate matrices $\mathcal{Q}_{\underline{Q}}$. Then $\mathcal{Q}_{\underline{Q}}$ is a non-empty, compact, and convex set of rate matrices that has separately specified rows, and for all $f\in\gamblesX$ it holds that $\underline{Q}f = \inf_{Q\in\mathcal{Q}_{\underline{Q}}} Qf$ and $\overline{Q}f = \sup_{Q\in\mathcal{Q}_{\underline{Q}}} Qf$.
	%\begin{equation*}
	%	\underline{Q}f = \inf_{Q\in\mathcal{Q}_{\underline{Q}}} Qf
	%	\quad\text{and}\quad
	%	\overline{Q}f = \sup_{Q\in\mathcal{Q}_{\underline{Q}}} Qf\,.
	%\end{equation*}
	Moreover, for all $f\in\gamblesX$ there is some $Q\in\mathcal{Q}_{\underline{Q}}$ such that $Qf=\underline{Q}f$, and there is some---possibly different---$Q\in\mathcal{Q}_{\underline{Q}}$ such that $Qf=\overline{Q}f$.	
\end{proposition}
Now fix any lower rate operator $\lrate$ and any $t\in\realsnonneg$, and let
\begin{equation}\label{eq:lower_rate_limit}
	e^{\lrate t}\coloneqq \lim_{n\to+\infty}\bigl(I+\nicefrac{t}{n}\lrate\bigr)^n\,.
\end{equation}
The operator $e^{\lrate t}$ is then a lower transition operator~\citep{de2017limit}, and the family $(e^{\lrate t})_{t\in\realsnonneg}$ is a semigroup of lower transition operators; it satisfies $e^{\lrate(t+s)}=e^{\lrate t}e^{\lrate s}$ for all $t,s\in\realsnonneg$, and $e^{\lrate 0}=I$. The analogous construction with an upper rate operator $\urate$ instead generates a semigroup $(e^{\urate t})_{t\in\realsnonneg}$ of upper transition operators.
When $\lrate$ and $\urate$ are taken with respect to the same set $\rateset$, these semigroups satisfy, for all $t\in\realsnonneg$, $f\in\gamblesX$, and $Q\in\rateset$,
\begin{equation}\label{eq:semigroup_domination}
	e^{\lrate t}f \leq e^{Qt}f \leq e^{\urate t}f\,.
\end{equation}
Here the importance again derives from the use as a computational tool; under the conditions of duality between $\rateset$ and $\lrate$, we have for any $f\in\gamblesX$ and any $t\in\realsnonneg$ that~\citep{skulj2015efficient,krak2017imprecise}
\begin{equation*}
	\underline{\mathbb{E}}^\mathrm{I}_{\mathcal{Q}}[f(X_t)\vert X_0=x] = \underline{\mathbb{E}}^\mathrm{M}_{\mathcal{Q}}[f(X_t)\vert X_0=x] = e^{\lrate t}f(x)\,.
\end{equation*}
Hence such inferences can be numerically computed by approximating~\eqref{eq:lower_rate_limit} with a finite choice of $n$, and then solving $n$ independent linear optimization problems over $\rateset$. Error bounds for this scheme are available in the literature~\citep{skulj2015efficient,krak2017imprecise,erreygers2021phd}.

\subsection{Class Structure}

Let us now fix a set $\rateset$ of rate matrices that we will use in the remainder of this work. Throughout, let $\lrate$ and $\urate$ denote the lower- and upper rate operators associated with $\rateset$. We impose several standard regularity conditions on this set: we assume that $\rateset$ is non-empty, compact, convex, and that it has separately specified rows. These are common assumptions that are imposed to ensure the duality between $\rateset$ and $\lrate$, which in turn guarantees that inferences with the induced imprecise-Markov chains remain well-behaved, as well as analytically (and, often, computationally) tractable.

We now have all the pieces to start studying the inference problem that is the subject of this work: the \emph{lower-} and \emph{upper expected hitting times} of the set $A\subset\states$ for \emph{continuous-time imprecise-Markov chains described by} $\rateset$.

Before we begin, let us impose two additional conditions on the dynamics of the system. 
\begin{assumption}\label{ass:absorbing}
	We assume that all states in $A$ are \emph{absorbing}, which is equivalent to requiring that $Q(x,x)=0$ for all $Q\in\rateset$ and all $x\in A$.
\end{assumption}
Note that this does not influence the inferences in which we are interested, since those only deal with behavior at times \emph{before} states in $A$ are reached. However, imposing this explicitly substantially simplifies the analysis.

Next, we assume that the set $A$ is \emph{lower reachable} from any state $x\in A^c$~\citep{de2017limit}. This means that we can construct a sequence $x_1,\ldots,x_{n+1}\in\states$ starting in any $x_1\in A^c$ and ending in some $x_{n+1}\in A$ such that, for all $k=1,\ldots,n$, it holds that $\lrate\ind{x_{k+1}}(x_k)>0$. This is equivalent~\citep{de2017limit} to
\begin{assumption}\label{ass:reachable}
	We assume $e^{\lrate t}\ind{A}(x)>0$ for all $t\in\realspos$ and all $x\in A^c$.
\end{assumption} 
Essentially, this means that for all elements of our imprecise-probabilistic models the probability of eventually hitting $A$ is bounded away from zero.
This ensures that the expected hitting times remain bounded for all $P\in\mathcal{P}^{\mathrm{I}}_{\rateset}$, so that we can ignore any extended real-valued analysis.  It also implies that for all $Q\in\rateset$ we have that $Q(x,x)\neq 0$ for all $x\in A^c$, which is relevant to meet the precondition of Proposition~\ref{prop:precise_cont_system}.
As a practical point, \citet{de2017limit} gives an algorithm to check whether a given set $\rateset$ satisfies this condition. 

On a technical level, Assumption~\ref{ass:reachable} is the crucial one for our results, and---unlike with Assumption~\ref{ass:absorbing}---it cannot really be ignored in practice. However, based on earlier work by~\citet{krak2019hitting} in the discrete-time setting, we hope in the future to strengthen our results to hold without this assumption.

\section{Subspace Dynamics}\label{subsec:subspace_dynamics}

In the context of hitting times, the interesting behavior of a process actually occurs \emph{before} it has reached a target state in~$A$. Hence it will be useful to introduce some machinery to study the transition dynamics as it relates to the states $A^c$.

To introduce the notation in a general way, choose any non-empty $\mathcal{Y}\subset\states$. Then for any $f\in\gamblesX$, let $f\vert_{\mathcal{Y}}\in\gamblesY{\mathcal{Y}}$ denote the restriction of $f$ to $\mathcal{Y}$. Conversely, for any $f\in\gamblesY{\mathcal{Y}}$, let $f\upX\in\gamblesX$ denote the unique extension of $f$ to $\states$ that satisfies $f(x)=0$ for all $x\in\states\setminus\mathcal{Y}$. Moreover, for any operator $M$ on $\gamblesX$, we define the operator $M\vert_{\mathcal{Y}}$ on $\gamblesY{\mathcal{Y}}$ as
\begin{equation*}
	M\vert_{\mathcal{Y}}f \coloneqq \bigl(M(f\upX)\bigr)\vert_{\mathcal{Y}}\quad\quad\text{for all $f\in\gamblesY{\mathcal{Y}}$.}
\end{equation*}
This somewhat verbose notation is perhaps most easily understood when $M$ is a linear operator, i.e. a matrix. In that case, $M\vert_{\mathcal{Y}}$ is simply the $\abs{\mathcal{Y}}\times\abs{\mathcal{Y}}$ sub-matrix of $M$ on the coordinates in $\mathcal{Y}$. The definition above allows us to extend this notion also to non-linear operators, and to lower- and upper transition and rate operators, specifically.

Now for any rate matrix $Q\in\rateset$, we call $G\coloneqq Q\resAc$ its corresponding \emph{subgenerator}. For any $t\in\realsnonneg$, we then define $e^{Gt}\coloneqq e^{Qt}\resAc$. We have the following result:
\begin{proposition}\label{prop:subsemigroup_precise}
	Fix $Q\in\rateset$ and let $G$ be its subgenerator. Then $e^{Gt} = \lim_{n\to+\infty} \bigl(I+\nicefrac{t}{n}G\bigr)^n$ for all $t\in\realsnonneg$.
	Moreover, the family $(e^{Gt})_{t\in\realsnonneg}$ is a semigroup.
\end{proposition}

Analogously, we define $\lsubgen\coloneqq\lrate\resAc$ and $\usubgen\coloneqq \urate\resAc$ to be the \emph{lower-} and \emph{upper subgenerators} corresponding to $\lrate$ and $\urate$, respectively. We also let $e^{\lsubgen t}\coloneqq e^{\lrate t}\resAc$ and $e^{\usubgen t}\coloneqq e^{\urate t}\resAc$. Perhaps unsurprisingly, we then have:
\begin{proposition}\label{prop:subsemigroup_imprecise}
	It holds that $e^{\lsubgen t} = \lim_{n\to+\infty} \bigl(I+\nicefrac{t}{n}\lsubgen\bigr)^n$ and $e^{\usubgen t} = \lim_{n\to+\infty} \bigl(I+\nicefrac{t}{n}\usubgen\bigr)^n$ for all $t\in\realsnonneg$. Moreover, the families $(e^{\lsubgen t})_{t\in\realsnonneg}$, $(e^{\usubgen t})_{t\in\realsnonneg}$ are semigroups.
\end{proposition}

Our Assumption~\ref{ass:reachable} implies the norm bound:
\begin{proposition}\label{prop:upper_subsemigroup_contractive}
	For any $t>0$, it holds that $\norm{\smash{e^{\usubgen t}}}<1$.
\end{proposition}

It is a straightforward consequence of the use of the supremum norm, together with Equation~\eqref{eq:semigroup_domination} and the fact that $e^{Qt}$ and $e^{\urate t}$ are (upper) transition operators, that also $\norm{\smash{e^{Gt}}} \leq \norm{\smash{e^{\usubgen t}}} <1$ for all $t\in\realspos$. Hence by the semigroup property we immediately have that $\lim_{t\to+\infty}\norm{\smash{e^{Gt}}}=0$. This also implies the following well-known result.
\begin{proposition}{\citep[Thm IV.1.4]{taylor1958introduction}}\label{prop:resolvent_existence}
	For any $Q\in\rateset$ with subgenerator $G$, and all $t>0$, the inverse operator $(I-e^{G t})^{-1}$ exists, and $(I-e^{G t})^{-1}=\sum_{k=0}^{+\infty}e^{Gtk}$.
\end{proposition}
This allows us to characterize hitting times for discrete-time homogeneous Markov chains whose transition matrix is given by $e^{Qt}$, as follows.
\begin{proposition}\label{prop:discrete_precise_by_inverse}
	Choose any $Q\in\rateset$, let $G$ be its subgenerator, and fix any $\Delta>0$. Let $P\in\mathbb{P}^{\mathrm{HM}}_{\natswith}$ be such that ${}^{P}T=e^{Q\Delta}$. Then the expected hitting times $h\coloneqq \mathbb{E}_P[\tau_{\natswith}\,\vert\,X_0]$ satisfy $h\resAc = (I-e^{G\Delta})^{-1}\ones$ and $h(x)=0$ for all $x\in A$.
\end{proposition}
\begin{proof}
	By Proposition~\ref{prop:precise_discr_system}, in $x\in A^c$ we have that
	\begin{equation*}
		h(x) = \ind{A^c}(x) + \ind{A^c}(x) e^{Q\Delta}h(x) = 1 + e^{Q\Delta}h(x)\,.
	\end{equation*}
	Conversely, it is immediate from the definition that $h(x)=0$ for all $x\in A$. This implies that $h=(h\resAc)\upX$, and hence
	\begin{equation*}
		h\resAc = \ones + \bigl(e^{Q\Delta}(h\resAc)\upX\bigr)\resAc = \ones + e^{G\Delta}h\resAc\,.
	\end{equation*}
	Re-ordering terms we have $(I - e^{G\Delta})h\resAc = \ones$. Now use Proposition~\ref{prop:resolvent_existence} and multiply with $(I-e^{G\Delta})^{-1}$.
\end{proof}

We need the following observation:
\begin{lemma}\label{lemma:subgen_negative_eigen}
	Consider any $Q\in\rateset$ with subgenerator $G$, and let $\sigma(G)$ be the set of eigenvalues of $G$. Then $\mathrm{Re}\,\lambda < 0$ for all $\lambda\in\sigma(G)$.
\end{lemma}

This implies that $0\notin\sigma(G)$, and so we have:
\begin{corollary}\label{cor:subgen_inverse}
	For any $Q\in\rateset$ with subgenerator $G$, the inverse operator $G^{-1}$ exists.
\end{corollary}

This allows us to characterize hitting times for continuous-time homogeneous Markov chains:
\begin{proposition}\label{prop:cont_precise_by_inverse}
	Choose any $Q\in\rateset$, let $G$ be its subgenerator, and let $P\in\mathbb{P}_{\realsnonneg}^{\mathrm{HM}}$ with $\rateforp=Q$. Then the expected hitting times $h\coloneqq \mathbb{E}_P[\tau_{\realsnonneg}\,\vert\,X_0]$ satisfy $h\resAc = -G^{-1}\ones$ and $h(x)=0$ for all $x\in A$.
\end{proposition}
\begin{proof}
	By Proposition~\ref{prop:precise_cont_system}, in $x\in A^c$ we have that
	\begin{equation*}
		-1 = -\ind{A^c}(x) = \ind{A^c}(x) Qh(x) = Qh(x)\,.
	\end{equation*}
	Conversely, it is immediate from the definition that $h(x)=0$ for all $x\in A$. This implies $h = (h\resAc)\upX$, and hence
	\begin{equation*}
		Gh\resAc = \bigl(Q(h\resAc)\upX)\bigr\resAc = (Qh)\resAc = -\ones\,.
	\end{equation*}
	Now use Corollary~\ref{cor:subgen_inverse} and multiply with $G^{-1}$.
\end{proof}

\subsection{Quasicontractivity of Subspace Dynamics}\label{sec:quasicontractive}

We already know from Proposition~\ref{prop:upper_subsemigroup_contractive} that $\norm{\smash{e^{\usubgen t}}}<1$ for all $t\in\realspos$. Since $\smash{e^{\usubgen 0}}=I$ (because it is a semigroup), it follows that $\norm{\smash{e^{\usubgen t}}}\leq 1$ for all $t\in\realsnonneg$. A semigroup that satisfies this property is said to be \emph{contractive}. Moreover, Proposition~\ref{prop:upper_subsemigroup_contractive} together with the semigroup property implies that $\lim_{t\to+\infty}\norm{\smash{e^{\usubgen t}}}=0$. A semigroup that satisfies this property is said to be \emph{uniformly exponentially stable}, and in such a case the following result holds:
\begin{proposition}\label{prop:subsemigroup_ues}
	There are $M\geq 1$ and $\xi>0$ such that $\norm{\smash{e^{\usubgen t}}} \leq M e^{-\xi t}$ for all $t\in\realsnonneg$.
\end{proposition}
This result means that the norm $\norm{\smash{e^{\usubgen t}}}$ decays exponentially as $t$ grows. However, for technical reasons we require an exponentially decaying norm bound with $M=1$; if this holds the semigroup is said to be \emph{quasicontractive}.

It is not clear that obtaining such a bound is possible when $\norm{\smash{e^{\usubgen t}}}$ is induced by the supremum norm $\norm{\cdot}$ on $\gamblesAc$. However, we can get it by defining a \emph{different} norm $\norm{\cdot}_*$ on $\gamblesAc$. We then obtain the quasicontractivity with respect to the induced operator norm $\norm{\cdot}_*$. Because $\gamblesAc$ is finite-dimensional these norms are equivalent, and such a result suffices for our purposes. This re-norming trick is originally due to~\citet{feller1953generation}, and an analogous construction is commonly used for semigroups of linear operators; see e.g.~\citep[Thm 12.21]{renardyrogers2004intropde}.

So, consider the $\xi>0$ from Proposition~\ref{prop:subsemigroup_ues}, and let
\begin{equation}\label{eq:alternative_norm}
	\norm{f}_* \coloneqq \sup_{t\in\realsnonneg} \norm{e^{\xi t}e^{\usubgen t}\abs{f}}\quad\text{for all $f\in\gamblesAc$,}
\end{equation} 
where $\abs{f}$ denotes the elementwise-absolute value of $f$.

\begin{proposition}\label{prop:newnorm_is_norm}
	The map $f\mapsto\norm{f}_*$ is a norm on $\gamblesAc$.
\end{proposition}
Moreover, we have the desired result:
\begin{proposition}\label{prop:renormed_quasicontractive}
	We have $\norm{\smash{e^{\usubgen t}}}_* \leq e^{-\xi t}$ for all $t\in\realsnonneg$.
\end{proposition}
Finally, the same bound holds for precise models:
\begin{proposition}\label{prop:precise_quasicontractive}
	For any $Q\in\rateset$ with subgenerator $G$ it holds that $\norm{e^{Gt}}_*\leq e^{-\xi t}$ for all $t\in\realsnonneg$.
\end{proposition}

\section{Hitting Times as Limits}\label{sec:hits_as_limits}

We now have all the pieces to explain the proof of our main results. The trick will be to establish a connection between hitting times for continuous-time imprecise-Markov chains, and hitting times for \emph{discrete}-time imprecise-Markov chains, for which analogous results were previously established by~\citet{krak2019hitting}. 

We essentially just look at a discretized continuous-time Markov chain taking steps of some fixed size $\Delta>0$, derive the expected hitting time for this discrete-time Markov chain, and then take the limit $\Delta\to 0^+$. The main difficulty is in establishing that this converges uniformly for all elements in our sets of processes; this is why we went through the trouble of establishing quasicontractivity in Section~\ref{sec:quasicontractive}. 

To start, for any $Q\in\rateset$ and $\Delta>0$, let $h^Q_\Delta$ be the minimal non-negative solution to the linear system\footnote{Note the re-scaled term $\Delta\ind{A^c}$ on the right-hand side, which distinguishes this from the system in Proposition~\ref{prop:precise_discr_system}; this is required since the hitting times for discrete-time Markov chains are expressed in the \emph{number} of steps, and to pass to continuous-time we need to measure the size of these steps.}
\begin{equation}
	h^Q_\Delta = \Delta \ind{A^c} + \ind{A^c} e^{Q\Delta}h^Q_\Delta\,,
\end{equation}
and let $h^Q$ be the minimal non-negative solution to
\begin{equation}
	\ind{A}h^Q = \ind{A^c} + \ind{A^c}Qh^Q\,.
\end{equation}
Then we know from Propositions~\ref{prop:precise_discr_system} and~\ref{prop:precise_cont_system} that $\nicefrac{1}{\Delta}h^Q_\Delta$ represents the expected hitting times of a discrete-time homogeneous Markov chain with transition matrix $e^{Q\Delta}$, and that $h^Q$ does the same for a continuous-time homogeneous Markov chain with rate matrix $Q$. We now have the following result:
\begin{proposition}\label{prop:precise_uniform_limit}
	There are $\delta>0$ and $L>0$ such that $\norm{h^Q_\Delta - h^Q}<\Delta L\norm{h^Q}$ for all $\Delta\in(0,\delta)$ and all $Q\in\rateset$.
\end{proposition}

Since $\norm{h^Q}$ is bounded due to Proposition~\ref{prop:cont_precise_by_inverse}:
\begin{corollary}\label{cor:precise_discretisation_converges}
	We have $\lim_{\Delta\to 0^+}h^Q_\Delta=h^Q$ for all $Q\in\rateset$.
\end{corollary}

We will now set up the analogous results for imprecise-Markov chains. First, let
\begin{equation}
	\underline{h}\coloneqq \inf_{Q\in\rateset} h^Q 
	\quad\text{and}\quad
	\overline{h}\coloneqq \sup_{Q\in\rateset} h^Q\,.
\end{equation}
Clearly, it follows from Proposition~\ref{prop:precise_cont_system} and the definition of lower- and upper expectations that these quantities represent the lower- and upper expected hitting times for the imprecise-Markov chain $\mathcal{P}_\rateset^{\mathrm{HM}}$, i.e. it holds that
\begin{equation*}
	\underline{h} = \underline{\mathbb{E}}_\rateset^{\mathrm{HM}}\bigl[\tau_{\realsnonneg}\,\vert\,X_0\bigr]
	\quad\text{and}\quad
	\overline{h} = \overline{\mathbb{E}}_\rateset^{\mathrm{HM}}\bigl[\tau_{\realsnonneg}\,\vert\,X_0\bigr]\,.
\end{equation*}
Now for any $\Delta>0$, let $\underline{h}_\Delta$ and $\overline{h}_\Delta$ denote the minimal non-negative solutions to the \emph{non-linear} systems
\begin{equation}\label{eq:lower_discretizes_system}
	\underline{h}_\Delta = \Delta\ind{A^c} + \ind{A^c} e^{\lrate \Delta}\underline{h}_\Delta
\end{equation}
and
\begin{equation}\label{eq:upper_discretizes_system}
	\overline{h}_\Delta = \Delta\ind{A^c} + \ind{A^c} e^{\urate \Delta}\overline{h}_\Delta\,.
\end{equation}
It was previously shown by~\citet{krak2019hitting} that---up to re-scaling with $\nicefrac{1}{\Delta}$---the quantities $\underline{h}_\Delta$ and $\overline{h}_\Delta$ represent the lower (resp. upper) expected hitting times of, identically, the discrete-time imprecise-Markov chains $\mathcal{P}_{\mathcal{T}_\Delta}^{\mathrm{HM}}$, $\mathcal{P}_{\mathcal{T}_\Delta}^{\mathrm{M}}$, and $\mathcal{P}_{\mathcal{T}_\Delta}^{\mathrm{I}}$ parameterized by the set $\mathcal{T}_{\Delta}$ of transition matrices that dominate $e^{\lrate\Delta}$. We now set out of prove an analogous result for continuous-time imprecise-Markov chains. We start with the following:
\begin{proposition}\label{prop:imprecise_limit}
	It holds that $\lim_{\Delta\to 0^+}\underline{h}_\Delta =  \underline{h}$ and $\lim_{\Delta\to 0^+}\overline{h}_\Delta = \overline{h}$.
\end{proposition}
This property allows us to leverage recent results by~\citet{erreygers2021phd} and~\citet{krak2021phd} regarding discrete and finite approximations of lower- and upper expectations in continuous-time imprecise-Markov chains, to obtain our first main result:
\begin{theorem}\label{thm:hitting_times_invariant}
	It holds that
	\begin{equation*}
		\underline{h} = \underline{\mathbb{E}}_\rateset^{\mathrm{HM}}\bigl[\tau_{\realsnonneg}\,\vert\,X_0\bigr] = \underline{\mathbb{E}}_\rateset^{\mathrm{M}}\bigl[\tau_{\realsnonneg}\,\vert\,X_0\bigr] = \underline{\mathbb{E}}_\rateset^{\mathrm{I}}\bigl[\tau_{\realsnonneg}\,\vert\,X_0\bigr]\,,
	\end{equation*}
	and, moreover, that
	\begin{equation*}
		\overline{h} = \overline{\mathbb{E}}_\rateset^{\mathrm{HM}}\bigl[\tau_{\realsnonneg}\,\vert\,X_0\bigr] = \overline{\mathbb{E}}_\rateset^{\mathrm{M}}\bigl[\tau_{\realsnonneg}\,\vert\,X_0\bigr] = \overline{\mathbb{E}}_\rateset^{\mathrm{I}}\bigl[\tau_{\realsnonneg}\,\vert\,X_0\bigr]\,.
	\end{equation*}
\end{theorem}
Moreover, it follows relatively straightforwardly from Proposition~\ref{prop:imprecise_limit} that the lower- and upper expected hitting times for continuous-time imprecise-Markov chains satisfy an immediate generalization of the system that characterizes the expected hitting times for (precise) continuous-time homogeneous Markov chains. This is our second main result:
\begin{theorem}\label{thm:lower_upper_hitting_system}
	Let $\underline{h}$ and $\overline{h}$ denote the lower- and upper expected hitting times for any one of $\mathcal{P}^\mathrm{HM}_\rateset$, $\mathcal{P}^\mathrm{M}_\rateset$, or $\mathcal{P}^\mathrm{I}_\rateset$. Then $\underline{h}$ is the minimal non-negative solution to the non-linear system $\ind{A}\underline{h}=\ind{A^c} + \ind{A^c}\lrate\,\underline{\vphantom{Q}h}$, and $\overline{h}$ is the minimal non-negative solution to the non-linear system $\ind{A}\overline{h}=\ind{A^c} + \ind{A^c}\urate\,\overline{h}$.
\end{theorem}

\section{Summary \& Conclusion}\label{sec:summary}

We have investigated the problem of characterizing expected hitting times for continuous-time imprecise-Markov chains. We have shown that under two relatively mild assumptions on the system's class structure---\emph{viz.} that the target states are absorbing, and can be reached by any non-target state---the corresponding lower (resp. upper) expected hitting time is the same for all three types of imprecise-Markov chains.

We have also demonstrated that these lower- and upper expected hitting times $\underline{h}$ and $\overline{h}$ satisfy the non-linear systems
\begin{equation*}
	\ind{A}\underline{h}=\ind{A^c} + \ind{A^c}\lrate\,\underline{\vphantom{Q}h}
	\quad\text{and}\quad
	\ind{A}\overline{h} = \ind{A^c} + \ind{A^c}\urate\,\overline{h}\,,
\end{equation*}
in analogy with the precise linear system~\eqref{eq:prop:precise_cont_system}. 
Indeed, we conclude that the lower- and upper expected hitting times for any of these three types of imprecise-Markov chains, can be fully characterized as the unique \emph{minimal} non-negative solutions to these respective systems.

We aim to strengthen these results in future work to hold with fewer assumptions on the system's class structure.
%
%\begin{contributions} % will be removed in pdf for initial submission,
%	% so you can already fill it to test with the
%	% ‘accepted’ class option
%	I did the thing.
%\end{contributions}

\section*{acknowledgements} % will be removed in pdf for initial submission,
	% so you can already fill it to test with the
	% ‘accepted’ class option
	
	We would like to sincerely thank Jasper De Bock for many stimulating discussions on the subject of imprecise-Markov chains, and for pointing out a technical error in an earlier draft of this work. We are also grateful for the constructive feedback of three anonymous reviewers.

\bibliography{krak_660}

\cleardoublepage

\appendix
% NOTE: necessary when ptmx or no mathfont class option is given
\renewcommand\thesection{\Alph{section}}

\section{Proofs and Lemmas for Section~\ref{subsec:subspace_dynamics}}

For certain operators, we note that subspace restriction distributes over operator composition:
\begin{lemma}\label{lemma:restriction_distributes}
	Let $M$ and $N$ be operators on $\gamblesX$ such that $N\vert_A = I$. Then $\bigl(MN\bigr)\resAc = M\resAc N\resAc$.
\end{lemma}
\begin{proof}
	Fix any $f\in\gamblesAc$. Then
	\begin{align*}
		M\resAc N\resAc f = M\Bigl( \bigl((Nf\upX)\resAc\bigr)\upX \Bigr)\resAc\,.
	\end{align*}
	Note that since $N\vert_A=I$ and $f\upX(x)=0$ for all $x\in A$, we also have $Nf\upX(x)=0$ for all $x\in A$. Hence in particular, it holds that $\bigl((Nf\upX)\resAc\bigr)\upX = Nf\upX$.
	%\begin{equation*}
	%	\bigl((Nf\upX)\resAc\bigr)\upX = Nf\upX\,.
	%\end{equation*}
	We therefore find that
	\begin{align*}
		M\resAc N\resAc f &= \bigl(MNf\upX\bigr)\resAc = (MN)\resAc f\,,
	\end{align*}
	which concludes the proof.
\end{proof}

This can be used in particular for certain operators associated with $Q\in\rateset$ and the associated lower- and upper rate operators:
\begin{lemma}\label{lemma:linear_approx_identity_on_A}
	Fix any $\Delta\geq 0$ and any $Q\in\rateset$. Then
	\begin{equation*}
		(I+\Delta Q)\vert_A = (I+\Delta \lrate)\vert_A = (I+\Delta \urate)\vert_A = I\,.
	\end{equation*}
\end{lemma}
\begin{proof}
	Fix any $Q\in\rateset$, and first choose any $f\in\gamblesX$ and $x\in A$. By Assumption~\ref{ass:absorbing} and the definition of rate matrices, we have $Q(x,y)=0$ for all $y\in\states$, whence $Qf(x)=\sum_{y\in\states}Q(x,y)f(y)=0$. Since $Q\in\rateset$ is arbitrary, we also have $\lrate f(x)=0$ and $\urate f(x)=0$. It follows that
	\begin{equation*}
		f(x) = (I+\Delta Q)f(x) = (I+\Delta \lrate)f(x)=(I+\Delta\urate)f(x)\,.
	\end{equation*}
	Since this is true for all $f\in\gamblesX$ and all $x\in A$, the result is now immediate.
\end{proof}
\begin{corollary}\label{cor:exponential_identity_on_A}
	For all $Q\in\rateset$ and $t\in\realsnonneg$ it holds that
	\begin{equation*}
		e^{Qt}\vert_A = e^{\lrate t}\vert_A = e^{\urate t}\vert_A = I\,.
	\end{equation*}
\end{corollary}
\begin{proof}
	Use Lemma~\ref{lemma:linear_approx_identity_on_A} and the definitions of $e^{Qt}$, $e^{\lrate t}$, $e^{\urate t}$.
\end{proof}

\begin{lemma}\label{lemma:restriction_norm_bound}
	Let $M$ and $N$ be operators on $\gamblesX$ such that $M\vert_A=I=N\vert_A$. Then $\norm{M\resAc - N\resAc} \leq \norm{M-N}$.
\end{lemma}
\begin{proof}
	Fix any $f\in\gamblesAc$ with $\norm{f}=1$. Then $\norm{f\upX}=1$. Moreover, since $f\upX(x)=0$ for all $x\in A$ and since $M\vert_A=I=N\vert_A$, we have that $(Mf\upX)(x)=0=(Nf\upX)(x)$ for all $x\in A$. Hence we find
	\begin{align*}
		&\norm{(M\resAc - N\resAc)f} \\
		&\quad\quad= \norm{\bigl((M-N)f\upX\bigr)\resAc} \\
		&\quad\quad= \norm{(M-N)f\upX} \\
		&\quad\quad\leq \sup\bigl\{\norm{(M-N)g}\,:\,g\in\gamblesX,\norm{g}=1\bigr\} \\
		&\quad\quad= \norm{M-N}\,.
	\end{align*}
	The result follows since $f\in\gamblesAc$ is arbitrary.
\end{proof}

\begin{proof}[Proof of Proposition~\ref{prop:subsemigroup_precise}]
	Fix $Q\in\rateset$ and let $G$ be its subgenerator. First fix any $t\in\realsnonneg$ and any $\epsilon>0$. Then by definition of $e^{Qt}$, for all $n\in\nats$ large enough it holds that
	\begin{equation*}
		\norm{e^{Qt} - \bigl(I+\nicefrac{t}{n}Q\bigr)^n } < \epsilon\,.
	\end{equation*}
	Moreover, by Lemmas~\ref{lemma:restriction_distributes} and~\ref{lemma:linear_approx_identity_on_A} we have
	\begin{equation*}
		\Bigl(\bigl(I+\nicefrac{t}{n}Q\bigr)^n\Bigr)\resAc = \bigl(I+\nicefrac{t}{n}G\bigr)^n\,,
	\end{equation*}
	and, by Corollary~\ref{cor:exponential_identity_on_A}, that $e^{Qt}\vert_A=I$. Hence by Lemma~\ref{lemma:restriction_norm_bound} we find
	\begin{equation*}
		\norm{e^{Gt} - \bigl(I+\nicefrac{t}{n} G\bigr)^n} \leq \norm{e^{Qt} - \bigl(I+\nicefrac{t}{n}Q\bigr)^n}<\epsilon\,.
	\end{equation*}
	Since $\epsilon>0$ is arbitrary, we have
	\begin{equation*}
		e^{Gt} = \lim_{n\to+\infty} \bigl(I+\nicefrac{t}{n} G\bigr)^n\,.
	\end{equation*}
	This concludes the proof of the first claim. 
	
	To see that $(e^{Gt})_{t\in\realsnonneg}$ is a semigroup, note that $(e^{Qt})_{t\in\realsnonneg}$ is a semigroup, then apply Lemma~\ref{lemma:restriction_distributes} and Corollary~\ref{cor:exponential_identity_on_A}.
\end{proof}

\begin{proof}[Proof of Proposition~\ref{prop:subsemigroup_imprecise}]
	The proof is completely analogous to the proof of Proposition~\ref{prop:subsemigroup_precise}; simply replace $Q$ with either $\lrate$ or $\urate$ as appropriate.
\end{proof}

\begin{proof}[Proof of Proposition~\ref{prop:upper_subsemigroup_contractive}]
	Let $\epsilon\coloneqq \min_{x\in A^c}e^{\lrate t}\ind{A}(x)$; then $\epsilon >0$ due to Assumption~\ref{ass:reachable}.
	Fix any $f\in\gamblesAc$ with $\norm{f}=1$. By definition, we have $e^{\usubgen t}f=e^{\urate t}\resAc f = \bigl(e^{\urate t}f\upX\bigr)\resAc$. 
	
	Let $\mathcal{T}_t$ denote the set of transition matrices that dominates $e^{\lrate t}$.
	Due to Proposition~\ref{prop:duality_trans}, there is some $T\in\mathcal{T}_t$ such that $Tf\upX = e^{\urate t}f\upX$. Fix any $x\in A^c$. Then, using that \mbox{$f\upX(y)=0$} for all $y\in A$, together with the fact that $T$ is a transition matrix, we have
	\begin{align*}
		\abs{Tf\upX(x)} &= \abs{\sum_{y\in\states} T(x,y)f\upX(y)} %\\
		%&= \abs{\sum_{y\in A^c} T(x,y)f\upX(y)} \\
		%&\leq \sum_{y\in A^c} T(x,y)\abs{f\upX(y)} \\
		%&
		\leq \sum_{y\in A^c} T(x,y)\,, %= T\ind{A^c}(x)\,.
	\end{align*}
	and hence $\abs{Tf\upX(x)}\leq T\ind{A^c}(x)$. We have $\ind{A}+\ind{A^c}=\ones$ and $T\ones(x)=1$ since $T$ is a transition matrix. Using the linear character of $T$, we find that
	\begin{align*}
		T\ind{A^c}(x) = T(\ones - \ind{A})(x) = 1 - T\ind{A}(x)\,.
	\end{align*}
	Since $T\in\mathcal{T}_t$ and $x\in A^c$ we have
	\begin{equation*}
		0<\epsilon = \min_{y\in A^c} e^{\lrate t}\ind{A}(y) \leq e^{\lrate t}\ind{A}(x) \leq T\ind{A}(x)\,.
	\end{equation*}
	Combining the above we find that
	\begin{equation*}
		\abs{Tf\upX(x)} \leq T\ind{A^c}(x)= 1 - T\ind{A}(x) \leq 1 - \epsilon\,.
	\end{equation*}
	Since this is true for all $x\in A^c$, we find that $\norm{(Tf\upX)\resAc}\leq 1-\epsilon$.  Moreover, since $Tf\upX=e^{\urate t}f\upX$, it follows that $\norm{(e^{\urate t}f\upX)\resAc}\leq 1-\epsilon$, or in other words, that
	\begin{equation*}
		\norm{e^{\usubgen t}f}\leq 1-\epsilon\,,\quad\quad\text{with $\epsilon>0$.}
	\end{equation*}
	The result follows since $f\in\gamblesAc$ with $\norm{f}=1$ is arbitrary.
\end{proof}

\begin{proof}[Proof of Lemma~\ref{lemma:subgen_negative_eigen}]
	Let $\rho(e^{G})\coloneqq \max_{\lambda\in\sigma(e^{G})}\abs{\lambda}$ denote the spectral radius of $e^{G}$. We know from Section~\ref{subsec:subspace_dynamics} that $\norm{e^G}<1$, and hence we have $\rho(e^G)\leq \norm{e^G}<1$~\citep[Thm V.3.5]{taylor1958introduction}. This implies that $\abs{\lambda}<1$ for all $\lambda\in\sigma(e^G)$.
	
	By the spectral mapping theorem~\citep[Lemma I.3.13]{engelnagel2000semigroupslinearee} we then have $e^{\mathrm{Re}\,\lambda}<1$ for all $\lambda\in\sigma(G)$, or in other words, that $\mathrm{Re}\,\lambda<0$ for all $\lambda\in\sigma(G)$.
\end{proof}

\section{Proofs and Lemmas for Section~\ref{sec:quasicontractive}}

\begin{proof}[Proof of Proposition~\ref{prop:subsemigroup_ues}]
	This proof is a straightforward generalization of an argument in~\citep[Prop I.3.12]{engelnagel2000semigroupslinearee}.

	Let first $q\coloneqq \norm{e^{\usubgen}}$; then $0<q<1$ due to Proposition~\ref{prop:upper_subsemigroup_contractive}. Define
	\begin{equation*}
		m\coloneqq \sup_{s\in [0,1]} \norm{e^{\usubgen s}}\,.
	\end{equation*}
	Then $m\geq 1$ since $m\geq \norm{e^{\usubgen 0}}=\norm{I}=1$. Moreover, $m\leq 1$ due to Proposition~\ref{prop:upper_subsemigroup_contractive}, and hence $m=1$. Now set $M\coloneqq \nicefrac{1}{q}$ and $\xi\coloneqq -\log q$; then $\xi >0$ since $q<1$.
	
	Fix any $t\in\realsnonneg$. If $t=0$ then the result is trivial, so let us suppose that $t>0$. Then there are $k\in\natswith$ and $s\in[0,1)$ such that $t=k+s$. Using the semigroup property, we have
	\begin{align*}
		\norm{e^{\usubgen t}} = \norm{e^{\usubgen(s+k)}} &\leq \norm{e^{\usubgen s}}\norm{e^{\usubgen}}^k \leq mq^k = e^{k\log q}\,.
	\end{align*}
	We have $k=t-s$ and $s\in[0,1)$, and so
	\begin{align*}
		\norm{e^{\usubgen t}} &\leq e^{k\log q} \\
		&= e^{(t-s)\log q} \\
		&= e^{t\log q}e^{-s\log q} \\
		&= e^{-\xi t}e^{-s\log q} \\
		&\leq e^{-\xi t}e^{-\log q} = \frac{1}{q}e^{-\xi t}=M e^{-\xi t}\,,
	\end{align*}
	which concludes the proof.
\end{proof}

\begin{proof}[Proof of Proposition~\ref{prop:newnorm_is_norm}] 
	It follows from the definition that for any upper transition operator $\overline{T}$ and any non-negative $f\in\gamblesX$, also $\overline{T}f$ is non-negative. In the sequel, we will therefore say that upper transition operators \emph{preserve non-negativity}. Since $e^{\urate t}$ is an upper transition operator, this property clearly extends also to $e^{\usubgen t}$.
	
	Now fix $f,g\in\gamblesAc$ and $t\in\realsnonneg$.  By preservation of non-negativity we have for any $x\in A^c$ that
	\begin{equation*}
		\abs{e^{\usubgen t}\abs{f+g}}(x) = e^{\usubgen t}\abs{f+g}(x)\,.
	\end{equation*}
	Moreover, we clearly have $\abs{f+g}\leq \abs{f}+\abs{g}$, and so by the monotonicity of upper transition operators, we have
	\begin{equation*}
		e^{\usubgen t}\abs{f+g}(x) \leq e^{\usubgen t}(\abs{f}+\abs{g})(x)\,.
	\end{equation*}
	Finally, by the subadditivity of upper transition operators, we find that
	\begin{equation*}
		e^{\usubgen t}(\abs{f}+\abs{g})(x) \leq e^{\usubgen t}\abs{f}(x) + e^{\usubgen t}\abs{g}(x)\,.
	\end{equation*}
	Again by preservation of non-negativity we have
	\begin{equation*}
		e^{\usubgen t}\abs{f}(x) + e^{\usubgen t}\abs{g}(x) = \abs{e^{\usubgen t}\abs{f}(x) + e^{\usubgen t}\abs{g}}(x)\,.
	\end{equation*}
	Because this is true for all $x\in A^c$, we find that
	\begin{align*}
		\norm{e^{\usubgen t}\abs{f+g}} &\leq \norm{e^{\usubgen t}\abs{f} + e^{\usubgen t}\abs{g}} \\
		&\leq \norm{e^{\usubgen t}\abs{f}} + \norm{e^{\usubgen t}\abs{g}}\,.
	\end{align*}
	Multiplying both sides with $e^{\xi t}$ and noting that $t\in\realsnonneg$ is arbitrary, we find that
	\begin{align*}
		\norm{f+g}_* &= \sup_{t\in\realsnonneg} \norm{e^{\xi t}e^{\usubgen t}\abs{f+g}} \\
		&\leq \sup_{t\in\realsnonneg} \norm{e^{\xi t}e^{\usubgen t}\abs{f}}+\norm{e^{\xi t}e^{\usubgen t}\abs{g}} \\
		&\leq \sup_{t\in\realsnonneg} \norm{e^{\xi t}e^{\usubgen t}\abs{f}} + \sup_{t\in\realsnonneg} \norm{e^{\xi t}e^{\usubgen t}\abs{g}} \\
		&= \norm{f}_* + \norm{g}_*\,.
	\end{align*}
	Hence we have established that $\norm{\cdot}_*$ satisfies the triangle inequality.
	
	Next, fix any $f\in\gamblesAc$ and $c\in\reals$. Then
	\begin{align*}
		\norm{cf}_* &= \sup_{t\in\realsnonneg} \norm{e^{\xi t}e^{\usubgen t}\abs{cf}} \\
		&= \sup_{t\in\realsnonneg} \norm{e^{\xi t}e^{\usubgen t}\abs{c}\abs{f}} \\
		&= \abs{c}\sup_{t\in\realsnonneg} \norm{e^{\xi t}e^{\usubgen t}\abs{f}} = \abs{c}\norm{f}_*\,.
	\end{align*}
	So $\norm{\cdot}_*$ is absolutely homogeneous.
	
	Finally, fix $f\in\gamblesAc$ and suppose that $\norm{f}_*=0$. It holds that
	\begin{align*}
		0 = \norm{f}_* \geq \norm{e^{\xi 0}e^{\usubgen 0}\abs{f}} \geq 0\,,
	\end{align*}
	whence it holds that $\norm{e^{\xi 0}e^{\usubgen 0}\abs{f}}=0$. This implies that also $\norm{e^{\usubgen 0}\abs{f}}=0$. Since $e^{\usubgen 0}=I$, we have
	\begin{align*}
		0 = \norm{e^{\usubgen 0}\abs{f}} = \norm{\abs{f}}=\norm{f}\,,
	\end{align*}
	whence $f=0$. Hence $\norm{\cdot}_*$ separates $\gamblesAc$.
\end{proof}

\begin{lemma}\label{lemma:newnorm_precise_bounded_upper}
	For any $Q\in\rateset$ with subgenerator $G$, any $f\in\gamblesAc$, and any $t\geq0$, it holds that $\norm{e^{Gt}f}_*\leq \norm{e^{\usubgen t}f}_*$.
\end{lemma}
\begin{proof}	
	Choose $f\in\gamblesAc$. Let $T$ be any matrix with non-negative entries. Then $\abs{Tf(x)}\leq \abs{T\abs{f}(x)}$ for all $x\in A^c$. In particular, we have
	\begin{align*}
		\abs{Tf(x)} &= \abs{\sum_{y\in A^c} T(x,y) f(y)} \\
		&\leq \sum_{y\in A^c} \abs{T(x,y)} \abs{f(y)} \\
		&= T\abs{f}(x) = \abs{T\abs{f}(x)}\,,
	\end{align*}
	where the final two equalities follow from the fact that $T$ only has non-negative entries. Since this is true for any matrix $T$ with non-negative entries, we have in particular that $\abs{e^{Gt}f}(x)\leq e^{Gt}\abs{f}(x)$. Similarly, it holds that
	\begin{align*}
		\abs{e^{\usubgen t}f}(x) &= \abs{\sup_{T\in\mathcal{T}_t} Tf(x) } \\
		&\leq \sup_{T\in\mathcal{T}_t} \abs{Tf(x) } \\
		&\leq\sup_{T\in\mathcal{T}_t} T\abs{f}(x) = e^{\usubgen t}\abs{f}(x)\,.
	\end{align*}
	
	It follows that, for any $s\in\realsnonneg$, we have
	\begin{align*}
		e^{\usubgen s}\abs{e^{Gt}f}(x) \leq e^{\usubgen s}e^{Gt}\abs{f}(x)\,.
	\end{align*}
	Due to preservation of non-negativity, and since this is true for any $x\in A^c$, we have
	\begin{equation*}
		\norm{e^{\usubgen s}\abs{e^{Gt}f}} \leq \norm{e^{\usubgen s}e^{Gt}\abs{f}}\,.
	\end{equation*}

	Now let $f\in\gamblesAc$ be such that $\norm{f}_*=1$ and $\norm{e^{Gt}}_*=\norm{e^{Gt}f}_*$; this $f$ clearly exists since $\gamblesAc$ is finite-dimensional. 
	Then we have
	\begin{align*}
		\norm{e^{Gt}}_* &= \norm{e^{Gt}f}_* \\
		&= \sup_{s\in\realsnonneg} \norm{e^{\xi s}e^{\usubgen s}\abs{e^{Gt}f}} \\
		&\leq \sup_{s\in\realsnonneg} \norm{e^{\xi s}e^{\usubgen s}e^{Gt}\abs{f}} = \norm{e^{Gt}\abs{f}}_* \leq \norm{e^{Gt}}_*\,,
	\end{align*}
	where the final inequality used that $\norm{\abs{f}}_*=\norm{f}_*=1$. Hence we have found that $\norm{e^{Gt}}_*=\norm{e^{Gt}\abs{f}}_*$.
	
	Since $e^{Qt}\in\mathcal{T}_t$ by Equation~\eqref{eq:semigroup_domination}, we also have
	\begin{equation*}
		e^{Gt}\abs{f} \leq e^{\usubgen t}\abs{f}\,.
	\end{equation*}
	By monotonicity of upper transition operators, this implies that
	\begin{equation*}
		e^{\usubgen s}e^{Gt}\abs{f} \leq e^{\usubgen s}e^{\usubgen t}\abs{f}
	\end{equation*}
	and, due to the preservation of non-negativity, we have
	\begin{equation*}
		e^{\usubgen s}e^{Gt}\abs{f}=\abs{e^{\usubgen s}e^{Gt}\abs{f}}\,,
	\end{equation*}
	and
	\begin{equation*}
		e^{\usubgen s}e^{\usubgen t}\abs{f}=\abs{e^{\usubgen s}e^{\usubgen t}\abs{f}}\,.
	\end{equation*}
	Hence for all $x\in A^c$ we have
	\begin{equation*}
		\abs{e^{\usubgen s}e^{Gt}\abs{f}}(x) \leq \abs{e^{\usubgen s}e^{\usubgen t}\abs{f}}(x)\,,
	\end{equation*}
	or in other words, that
	\begin{equation*}
		\norm{e^{\usubgen s}e^{Gt}\abs{f}} \leq \norm{e^{\usubgen s}e^{\usubgen t}\abs{f}}\,.
	\end{equation*}
	
	Since this holds for all $s\in\realsnonneg$, we have
	\begin{align*}
		\norm{e^{Gt}}_* = \norm{e^{Gt}\abs{f}}_* &= \sup_{s\in\realsnonneg} \norm{e^{\xi s}e^{\usubgen s}e^{Gt}\abs{f}} \\
		&\leq \sup_{s\in\realsnonneg} \norm{e^{\xi s}e^{\usubgen s}e^{\usubgen t}\abs{f}} \\
		&= \norm{e^{\usubgen t}\abs{f}}_* \leq \norm{e^{\usubgen t}}_*\,,
	\end{align*}
	which concludes the proof.
\end{proof}

\begin{proof}[Proof of Proposition~\ref{prop:renormed_quasicontractive}]
	The argument is analogous to the well-known case for linear quasicontractive semigroups; for a similar result, see e.g.~\citep[Thm 12.21]{renardyrogers2004intropde}. 
	So, fix any $t\in\realsnonneg$ and $f\in\gamblesAc$.
	
	Using a similar argument as used in the proof of Lemma~\ref{lemma:newnorm_precise_bounded_upper}, we use the preservation of non-negativity and the monotonicity of upper transition operators, to find for any $s\in\realsnonneg$ that
	\begin{equation*}
		\norm{e^{\xi s}e^{\usubgen s}\abs{e^{\usubgen t}f}} \leq \norm{e^{\xi s}e^{\usubgen s}e^{\usubgen t}\abs{f}}\,.
	\end{equation*}
	Hence we have
	\begin{align*}
		\norm{e^{\usubgen t}f}_* &= \sup_{s\in\realsnonneg}\norm{e^{\xi s}e^{\usubgen s}\abs{e^{\usubgen t}f}} \\
		&\leq \sup_{s\in\realsnonneg}\norm{e^{\xi s}e^{\usubgen s}e^{\usubgen t}\abs{f}} \\
		&= e^{-\xi t}e^{\xi t}\sup_{s\in\realsnonneg}\norm{e^{\xi s}e^{\usubgen (s+t)}\abs{f}} \\
		&= e^{-\xi t}\sup_{s\in\realsnonneg}\norm{e^{\xi (s+t)}e^{\usubgen (s+t)}\abs{f}} \\
		&= e^{-\xi t}\sup_{s\in\reals_{\geq t}}\norm{e^{\xi (s)}e^{\usubgen s}\abs{f}} \leq e^{-\xi t} \norm{f}_*\,,
	\end{align*}
	where for the second equality we used the semigroup property. Since $f\in\gamblesAc$ is arbitrary, this implies that
	\begin{equation*}
		\norm{e^{\usubgen t}}_* = \sup\bigl\{ \norm{e^{\usubgen t}f}_*\,:\,f\in\gamblesAc, \norm{f}_*=1\bigr\} \leq e^{-\xi t}\,,
	\end{equation*}
	which completes the proof.
\end{proof}

\begin{proof}[Proof of Proposition~\ref{prop:precise_quasicontractive}]
	This is immediate from Lemma~\ref{lemma:newnorm_precise_bounded_upper} and Proposition~\ref{prop:renormed_quasicontractive}.
\end{proof}

\section{Proofs and Lemmas for Section~\ref{sec:hits_as_limits}}

The following result is well-known, but we state it here for convenience:
\begin{lemma}\label{lemma:geometric_inverse_bound}
	Let $T$ be a linear bounded operator on a Banach space with norm $\norm{\cdot}_*$. Suppose that $\norm{T}_*<1$ and that $(I-T)^{-1}$ exists. Then
	\begin{equation*}
		\norm{(I-T)^{-1}}_* \leq \frac{1}{1-\norm{T}_*}\,.
	\end{equation*}
\end{lemma}
\begin{proof}
	Since $\norm{T}_*<1$ we have $(I-T)^{-1}=\sum_{k=0}^{+\infty}T^k$. Taking norms,
	\begin{equation*}
		\norm{(I-T)^{-1}}_* = \norm{\sum_{k=0}^{+\infty}T^k}_* \leq \sum_{k=0}^{+\infty}\norm{T}_*^k = \frac{1}{1-\norm{T}_*}\,,
	\end{equation*}
	where the final step used the value of the geometric series and that $\norm{T}_*<1$.
\end{proof}

\begin{lemma}\label{lemma:resolvent_uniform_bound}
	There is some $C> 0$ such that for any $\Delta>0$ with $\Delta\xi<1$, and any $Q\in\rateset$ with subgenerator $G$, it holds that $\norm{(I-e^{G\Delta})^{-1}} < \nicefrac{C}{\Delta}$.
\end{lemma}
\begin{proof}
	Let $\xi>0$ be as in Proposition~\ref{prop:subsemigroup_ues}, and let $\norm{\cdot}_*$ be the norm from Equation~\eqref{eq:alternative_norm}. Since $\gamblesAc$ is finite-dimensional the norms $\norm{\cdot}$ and $\norm{\cdot}_*$ are equivalent, and hence there is some $c>0$ such that $\norm{f}\leq c\norm{f}_*$ for all $f\in\gamblesAc$. Set $C\coloneqq \nicefrac{2c}{\xi}$; then $C>0$ since $\xi>0$.
	
	Fix any $\Delta>0$ such that $\Delta\xi<1$, and any $Q\in\rateset$ with subgenerator $G$. It follows from Proposition~\ref{prop:precise_quasicontractive} that $\norm{e^{G\Delta}}_* \leq e^{-\xi\Delta}$. Using a standard quadratic bound on the negative scalar exponential, we have
	\begin{equation}\label{eq:lemma:resolvent_bound:exponential_bound}
		\norm{e^{G\Delta}}_* \leq e^{-\xi \Delta} \leq 1 - \xi\Delta + \frac{1}{2}\Delta^2\xi^2 < 1-\frac{\Delta\xi}{2} < 1\,,
	\end{equation}
	where the third inequality used that $\Delta\xi<1$.
	Notice that $\norm{e^{G\Delta}}_*\leq e^{-\xi \Delta}<1$. Moreover, $(I-e^{G\Delta})^{-1}$ exists by Proposition~\ref{prop:resolvent_existence}. By the norm equivalence, we have
	\begin{equation}\label{eq:lemma:resolvent_bound:norm_equiv}
		\norm{(I-e^{G\Delta})^{-1}} \leq c\norm{(I-e^{G\Delta})^{-1}}_*\,,
	\end{equation}	
	and, by Lemma~\ref{lemma:geometric_inverse_bound}, that
	\begin{equation*}
		\norm{(I-e^{G\Delta})^{-1}}_* \leq \frac{1}{1 - \norm{e^{G\Delta}}_*}\,.
	\end{equation*}
	Using Equation~\eqref{eq:lemma:resolvent_bound:exponential_bound} we obtain
	\begin{equation*}
		\norm{(I-e^{G\Delta})^{-1}}_* \leq \frac{1}{1 - \norm{e^{G\Delta}}_*} < \frac{1}{1 - 1 + \frac{\Delta\xi}{2}} = \frac{1}{\Delta}\frac{2}{\xi}\,.
	\end{equation*}
	Combining with Equation~\eqref{eq:lemma:resolvent_bound:norm_equiv} yields
	\begin{equation*}
		\norm{(I-e^{G\Delta})^{-1}} < c\frac{1}{\Delta}\frac{2}{\xi} = \frac{C}{\Delta} \,,
	\end{equation*}
	which concludes the proof.
\end{proof}

\begin{proof}[Proof of Proposition~\ref{prop:precise_uniform_limit}]
	Let $\xi,C>0$ be as in Lemma~\ref{lemma:resolvent_uniform_bound}, and let $\delta\coloneqq\nicefrac{1}{\xi}$ and $L \coloneqq C\norm{\rateset}^2$	with $\norm{\rateset}\coloneqq \sup_{Q\in\rateset}\norm{Q}$;  note that $\norm{\rateset}\in\realsnonneg$ since $\rateset$ is bounded by assumption. Observe that we must have $\norm{\rateset}>0$ due to Assumption~\ref{ass:reachable}, whence $L>0$.
	
	Choose any $\Delta\in(0,\delta)$ and $Q\in\rateset$.
	It is immediate from the definitions that $h^Q(x)=0=h^Q_\Delta(x)$ for all $x\in A$ and all $Q\in\rateset$, so it remains to bound the norm on $A^c$. 
	
	Let $G$ be the subgenerator of $Q$ on $A^c$. By Proposition~\ref{prop:discrete_precise_by_inverse} we have that $h^Q_\Delta\vert_{A^c} = (I - e^{G\Delta})^{-1}\Delta\ones$. Using the definition of $h^Q$ this implies that
	\begin{align*}
		h^Q_\Delta\vert_{A^c} - e^{G\Delta}h^Q_\Delta\vert_{A^c} = \Delta\ones = -\Delta Gh^Q\resAc\,.
	\end{align*}
	Re-ordering terms we have
	\begin{equation*}
		h^Q_\Delta\vert_{A^c} = e^{G\Delta}h^Q_\Delta\vert_{A^c} -\Delta Gh^Q\vert_{A^c}\,.
	\end{equation*}
	Let $B = e^{G\Delta} - (I+\Delta G)$.
	We find that
	\begin{align*}
		h^Q_\Delta&\vert_{A^c} - h^Q\vert_{A^c} \\
		&= e^{G\Delta}h^Q_\Delta\vert_{A^c} -\Delta Gh^Q\vert_{A^c} - h^Q\vert_{A^c} \\
		&= e^{G\Delta}h^Q_\Delta\vert_{A^c} - (I+\Delta G)h^Q\vert_{A^c} \\
		%&=  e^{G\Delta}h_\Delta\vert_{A^c} - e^{G\Delta}h\vert_{A^c} + e^{G\Delta}h\vert_{A^c} - (I+\Delta G)h\vert_{A^c} \\
		&= e^{G\Delta}(h^Q_\Delta\vert_{A^c} - h^Q\vert_{A^c}) + \bigl(e^{G\Delta} - (I+\Delta G)\bigr)h^Q\vert_{A^c} \\
		&= e^{G\Delta}(h^Q_\Delta\vert_{A^c} - h^Q\vert_{A^c}) + Bh^Q\vert_{A^c}\,.
	\end{align*}
	We see that the difference on the left-hand side occurs again on the right-hand side. Hence we can substitute the same expansion $n\in\nats$ times to get
	\begin{align*}
		&h^Q_\Delta\vert_{A^c} - h^Q\vert_{A^c} \\
		&\quad= e^{G\Delta (n+1)}(h^Q_\Delta\vert_{A^c} - h^Q\vert_{A^c}) + \sum_{k=0}^n e^{G\Delta k} Bh^Q\vert_{A^c}\,.
	\end{align*}
	Since we know from Section~\ref{subsec:subspace_dynamics} that $\lim_{t\to+\infty} e^{Gt}=0$, we see that the left summand vanishes as we take \mbox{$n\to+\infty$} and, using Proposition~\ref{prop:resolvent_existence}, we have $(I-e^{Q\Delta})^{-1}=\sum_{k=0}^{+\infty} e^{G\Delta k}$. So, passing to this limit and taking norms, we find
	\begin{align*}
		\norm{h^Q_\Delta\vert_{A^c} - h^Q\vert_{A^c}} &= \norm{(I-e^{G\Delta})^{-1}Bh^Q\vert_{A^c}} \\
		&\leq \norm{(I-e^{G\Delta})^{-1}}\norm{B}\norm{h^Q\vert_{A^c}}\,.
	\end{align*}
	Using Lemmas~\ref{lemma:linear_approx_identity_on_A} and \ref{lemma:restriction_norm_bound} and Corollary~\ref{cor:exponential_identity_on_A}, we have
	\begin{equation*}
		\norm{B} = \norm{e^{G\Delta}-(I+\Delta G)} \leq \norm{e^{Q\Delta}-(I+\Delta Q)}\,,
	\end{equation*}
	and so, due to~\citep[Lemma B.8]{krak2021phd}, we have $\norm{B}\leq \Delta^2\norm{Q}^2$. Since $Q\in\rateset$ it follows that $\norm{Q}\leq\norm{\rateset}$, and so $\norm{B}\leq \Delta^2\norm{\rateset}^2$. Since $\Delta<\delta$ we have $\Delta\xi<1$, whence $\norm{(I-e^{G\Delta})^{-1}}<\nicefrac{C}{\Delta}$ due to Lemma~\ref{lemma:resolvent_uniform_bound}. In summary we find
	\begin{align*}
		\norm{h^Q_\Delta\vert_{A^c} - h^Q\vert_{A^c}} &<\frac{C}{\Delta}\Delta^2\norm{\rateset}^2\norm{h^Q\vert_{A^c}} = \Delta L\norm{h^Q}\,,
	\end{align*}
	which concludes the proof.
\end{proof}

\begin{proposition}{\citep[Prop 7]{krak2020computing}}\label{prop:discrete_lower_by_limit_of_transmat}
	Fix any $\Delta>0$, and let $\mathcal{T}_\Delta$ denote the set of transition matrices that dominate $e^{\lrate\Delta}$. Choose any $T_0\in\mathcal{T}_\Delta$. For all $n\in\natswith$, let $h_n$ be the (unique) non-negative solution to $h_n=\Delta\ind{A^c} + \ind{A^c}T_nh_n$, and let $T_{n+1}\in\mathcal{T}_\Delta$ be such that $T_{n+1}h_n=e^{\lrate\Delta}h_n$.
	
	% Let $h_0\coloneqq \underline{h}_\Delta$, and choose any $T_1\in\mathcal{T}_\Delta$ such that $e^{\lrate\Delta}h_0=T_1h_0$. Then, for all $n\in\nats$, let $h_{n}$ be the unique solution to $h_n=\Delta\ind{A^c}+\ind{A^c}\cdot T_nh_n$, and choose any $T_{n+1}\in\mathcal{T}_\Delta$ such that $e^{\lrate\Delta}h_n=T_{n+1}h_n$. 
	
	Then $\lim_{n\to+\infty}h_n=\underline{h}_\Delta$.
\end{proposition}
\begin{proof}
	The preconditions of the reference actually require every $T_n$ to be an extreme point of $\mathcal{T}_{\Delta}$, but inspection of the proof of~\citep[Prop 7]{krak2020computing} shows that this is not required; the superfluous condition is only used to streamline the statement of an algorithmic result further on in that work.
\end{proof}

We next need some results that involve transition matrices ${}^{P}T_t^s$ corresponding to (not-necessarily homogeneous) Markov chains $P\in\mathbb{P}^{\mathrm{M}}_\rateset$. We recall from Section~\ref{subsec:precise_dynamics} that these are defined for any $t,s\in\realsnonneg$ with $t\leq s$ as
\begin{equation*}
	{}^{P}T_t^s(x,y) \coloneqq P(X_s=y\,\vert\,X_t=x)\quad\text{for all $x,y\in\states$.}
\end{equation*}

\begin{lemma}\label{lemma:cosequence_markov_transmat}
	Consider the sequence $(h_n)_{n\in\natswith}$ constructed as in Proposition~\ref{prop:discrete_lower_by_limit_of_transmat}. For any $n\in\natswith$, there is a Markov chain $P_{n+1}\in\mathcal{P}_\rateset^\mathrm{M}$ with corresponding transition matrix ${}^{(n+1)}T_0^\Delta$ such that ${}^{(n+1)}T_0^{\Delta}h_n = e^{\lrate \Delta}h_n$. 
	
	Hence in particular, we can choose the co-sequence $(T_n)_{n\in\nats}$ in Proposition~\ref{prop:discrete_lower_by_limit_of_transmat} to be $({}^{(n)}T_0^{\Delta})_{n\in\nats}$. 
\end{lemma}
\begin{proof}
	This follows from~\citep[Cor 6.24]{krak2021phd} and the fact that $\rateset$ is non-empty, compact, convex, and has separately specified rows.
\end{proof}

\begin{proposition}\label{prop:lower_discrete_time_by_markov}
	For all $\Delta>0$ there is a Markov chain $P\in\mathcal{P}_\rateset^{\mathrm{M}}$ with corresponding transition matrix $T={}^{P}T_{0}^\Delta$, such that the unique solution $h$ to $h=\Delta\ind{A^c}+\ind{A^c} Th$ satisfies $h=\underline{h}_\Delta$.
\end{proposition}
\begin{proof}
	Let $\mathcal{T}_\Delta^\mathrm{M}\coloneqq \{{}^PT_0^\Delta\,:\,P\in\mathcal{P}_\rateset^{\mathrm{M}}\}$, and	let $(h_n)_{\in\nats}$ be as in Proposition~\ref{prop:discrete_lower_by_limit_of_transmat}, with the co-sequence $(T_n)_{n\in\nats}$ chosen as in Lemma~\ref{lemma:cosequence_markov_transmat} to consist of transition matrices corresponding to Markov chains in $\mathcal{P}_\rateset^\mathrm{M}$. Then $(T_n)_{n\in\nats}$ lives in $\mathcal{T}_\Delta^\mathrm{M}$.
	
	The set $\mathcal{T}_{\Delta}^\mathrm{M}$ is compact by~\citep[Cor 5.18]{krak2021phd} and the fact that $\rateset$ is non-empty, compact, and convex. Hence we can find a subsequence $(T_{n_j})_{j\in\nats}$ with $\lim_{j\to+\infty} T_{n_j}=:T\in\mathcal{T}_\Delta^\mathrm{M}$. Since $T\in\mathcal{T}_\Delta^\mathrm{M}$, there is a Markov chain $P\in\mathcal{P}_\rateset^\mathrm{M}$ with corresponding transition matrix $T={}^{P}T_{0}^\Delta$.
	
	Moreover, since $T,T_{n_j}\in\mathcal{T}_\Delta^\mathrm{M}$, it follows from~\citep[Cor 6.24]{krak2021phd} that the transition matrices $T$ and all $T_{n_j}$ dominate the lower transition operator $e^{\lrate \Delta}$. Together with Assumption~\ref{ass:reachable}, this allows us to invoke~\cite[Prop 6]{krak2020computing}, by which we can let $h$ be the unique solution to $h=\Delta\ind{A^c}+\ind{A^c} Th$, and it holds for any $j\in\nats$ that $h_{n_j}\vert_A=0$, and
	\begin{equation*}
		h_{n_j}\resAc=(I-T_{n_j}\resAc)^{-1}\ones\Delta\,.
	\end{equation*}
	Similarly, it holds that $h\vert_A=0$, and
	\begin{equation*}
		h\resAc = (I-T\resAc)^{-1}\ones\Delta\,.
	\end{equation*}
	Since $\lim_{h\to+\infty}T_{n_j}=T$ and by continuity of the map $M\mapsto (I-M)^{-1}$---which holds since all these inverses exist---it follows that $h\resAc = \lim_{j\to+\infty}h_{n_j}\resAc$. Since also $h\vert_A=h_{n_j}\vert_A$, it follows that $\lim_{j\to+\infty}h_{n_j}=h$.
	
	By Proposition~\ref{prop:discrete_lower_by_limit_of_transmat} we have $\lim_{n\to+\infty} h_n=\underline{h}_\Delta$, and hence we conclude that $\underline{h}_\Delta = \lim_{j\to+\infty} h_{n_j} = h$.
\end{proof}

\begin{proposition}\label{prop:markov_transmat_as_product}
	Fix any $t\geq 0$ and consider any Markov chain $P\in\mathcal{P}_\rateset^\mathrm{M}$ with transition matrix ${}^{P}T_0^t$. Choose any $\epsilon>0$. Then there is some $m\in\nats$ such that for all $n\geq m$ there are $Q_1,\ldots,Q_n\in\rateset$, such that
	\begin{equation*}
		\norm{{}^{P}T_0^t - \prod_{i=1}^n(I+\nicefrac{t}{n}Q_i)} < \epsilon\,.
	\end{equation*}
\end{proposition}
\begin{proof}
	The result is trivial if $t=0$, so let us consider the case where $t>0$. Let $\epsilon'\coloneqq \nicefrac{\epsilon}{2t}$. By~\citep[Lemma 5.12]{krak2021phd} there is some $m\in\nats$ such that for all $n\geq m$ and with $\Delta\coloneqq\nicefrac{t}{n}$, for all $i=1,\ldots,n$ there is some $Q_i\in\rateset$ such that
	\begin{equation*}
		\norm{{}^{P}T_{(i-1)\Delta}^{i\Delta} - (I+\Delta Q_i)} \leq \Delta\epsilon'\,.
	\end{equation*}
	Since $P$ is a Markov chain, we can factor its transition matrices~\citep[Prop 5.1]{krak2021phd} as
	\begin{equation*}
		{}^{P}T_0^t = {}^{P}T_0^\Delta{}^{P}T_\Delta^{2\Delta}\cdots {}^{P}T_{t-\Delta}^{t} = \prod_{i=1}^n {}^{P}T_{(i-1)\Delta}^{i\Delta}\,.
	\end{equation*}
	Using~\citep[Lemma B.5]{krak2021phd} for the first inequality, we have
	\begin{align*}
		&\norm{{}^{P}T_0^t - \prod_{i=1}^n(I+\Delta Q_i)} \\
		&= \norm{\prod_{i=1}^n {}^{P}T_{(i-1)\Delta}^{i\Delta} - \prod_{i=1}^n(I+\Delta Q_i)} \\
		&\leq \sum_{i=1}^n \norm{{}^{P}T_{(i-1)\Delta}^{i\Delta} - (I+\Delta Q_i)} \\
		&\leq \sum_{i=1}^n \Delta\epsilon' = n\frac{t}{n}\frac{\epsilon}{2t} = \frac{\epsilon}{2}\,,
	\end{align*}
	which concludes the proof.
\end{proof}

\begin{lemma}\label{lemma:limit_of_Q_limit_of_h}
	Consider a sequence $(Q_n)_{n\in\nats}$ in $\rateset$ with limit $Q_*\coloneqq \lim_{n\to+\infty} Q_n$. For all $n\in\nats$, let $h_n$ denote the minimal non-negative solution to $\ind{A}h_n=\ind{A^c}+\ind{A^c} Q_nh_n$, and let $h_*$ denote the minimal non-negative solution to $\ind{A}h_*=\ind{A^c} + \ind{A^c} Q_*h_*$. Then $h_*=\lim_{n\to+\infty}h_n$.
\end{lemma}
\begin{proof}
	Since $\rateset$ is closed, we have $Q_*\in\rateset$.
	Let $(G_n)_{n\in\nats}$ and $G_*$ denote the subgenerators of $(Q_n)_{n\in\nats}$ and $Q_*$, respectively. Then $G_*^{-1}$ and $G_n^{-1}$, $n\in\nats$ exist by Corollary~\ref{cor:subgen_inverse}, and hence we also have $\lim_{n\to+\infty} G_n^{-1}=G_*^{-1}$. Right-multiplying with $-\ones$ and applying Proposition~\ref{prop:cont_precise_by_inverse} gives
	\begin{equation*}
		\lim_{n\to+\infty} h_n\resAc = \lim_{n\to+\infty} -G_n^{-1}\ones = -G_*^{-1}\ones = h_*\resAc\,.
	\end{equation*}
	Finally, by definition we trivially have $h_n(x)=0=h_*(x)$ for all $x\in A$. Hence also $\lim_{n\to+\infty}h_n\vert_A = h_*\vert_A$.
\end{proof}

\begin{lemma}{\citep[Cor 13]{krak2019hitting}}\label{lemma:discrete_infimum_bound}
	Fix any $\Delta>0$ and let $\underline{h}_\Delta$ be the minimal non-negative solution to the non-linear system~\eqref{eq:lower_discretizes_system}. Let $\mathcal{T}_\Delta$ denote the set of transition matrices that dominate $e^{\lrate \Delta}$ and, for all $T\in\mathcal{T}_{\Delta}$, let $h_T$ denote the minimal non-negative solution to the linear system $h_T=\Delta\ind{A^c} + \ind{A^c}Th_T$. Then it holds that
	\begin{equation*}
		\underline{h}_\Delta = \inf_{T\in\mathcal{T}_\Delta} h_T\,.
	\end{equation*}
\end{lemma}

\begin{proof}[Proof of Proposition~\ref{prop:imprecise_limit}]
	We only give the proof for the lower hitting times, i.e. that $\lim_{\Delta\to 0^+}\norm{\underline{h}_\Delta - \underline{h}}=0$. The argument for the upper hitting times is completely analogous.
	
	Choose any two sequences $(\Delta_n)_{n\in\nats}$ and $(\epsilon_n)_{n\in\nats}$ in $\realspos$ such that $\lim_{n\to+\infty} \Delta_n=0$ and $\lim_{n\to+\infty} \epsilon_n=0$. We will assume without loss of generality that $\Delta_n\norm{\rateset}\leq 1$ for all $n\in\nats$, where $\norm{\rateset}=\sup_{Q\in\rateset}\norm{Q}$. 
	
	Now first fix any $n\in\nats$, and consider $\underline{h}_{\Delta_n}$. By Proposition~\ref{prop:lower_discrete_time_by_markov} there is a Markov chain $P_n\in\mathcal{P}_\rateset^\mathrm{M}$ with transition matrix $T_n\coloneqq {}^{P_n}T_0^{\Delta_n}$ such that the unique solution $h_n$ to $h_n=\Delta_n\ind{A^c}+\ind{A^c} T_nh_n$ satisfies $h_n=\underline{h}_{\Delta_n}$.
	
	By Proposition~\ref{prop:markov_transmat_as_product}, there are $m_n\in\nats$ with $m_n\geq n$ and $Q^{(n)}_1,\ldots,Q^{(n)}_{m_n}$ in $\rateset$ such that, with
	\begin{equation*}
		\Phi_n \coloneqq \prod_{i=1}^{m_n}\left(I+\frac{\Delta_n}{m_n}Q_i^{(n)}\right)\,,
	\end{equation*}
	it holds that $\norm{T_n - \Phi_n} < \epsilon_n$. Now define
	\begin{equation*}
		Q_n \coloneqq \sum_{i=1}^{m_n} \frac{1}{m_n} Q_i^{(n)}\,.
	\end{equation*}
	Then $Q_n\in\rateset$ since $\rateset$ is convex. Let $h_{Q_n}$ denote the minimal non-negative solution to $\ind{A}h_{Q_n}=\ind{A^c} + \ind{A^c} Q_n h_{Q_n}$.
	
	By repeating this construction for all $n\in\nats$, we obtain a sequence $(Q_n)_{n\in\nats}$ in $\rateset$. Since $\rateset$ is (sequentially) compact, we can consider a subsequence $(Q_{n_j})_{j\in\nats}$ such that $\lim_{j\to+\infty} Q_{n_j} =: Q_* \in\rateset$. 
	
	Let $h_*$ be the minimal non-negative solution to $\ind{A}h_*=\ind{A^c}+\ind{A^c} Q_*h_*$. We now need to estimate some norm bounds that hold by choosing $j$ large enough. Let $K=5$ and fix any $\delta>0$.
	
	Since $(Q_{n_j})_{j\in\nats}$ converges to $Q_*$, it follows from Lemma~\ref{lemma:limit_of_Q_limit_of_h} that for $j$ large enough, we have
	\begin{equation}\label{eq:prop:imprecise_limit:step_1}
		\norm{h_{Q_{n_j}} - h_*} < \frac{\delta}{K}
	\end{equation}
	Since $h_*$ is bounded, this also implies that the sequence $(h_{Q_{n_j}})_{j\in\nats}$ is eventually uniformly bounded above in norm by some constant $M\geq 0$, say.

	For all $j\in\nats$, let $\hat{h}_{n_j}$ be such that $\hat{h}_{n_j}\resAc\coloneqq (I-e^{G_{n_j}\Delta_{n_j}})^{-1}\Delta_{n_j}\ones$ and $\hat{h}_{n_j}\vert_A\coloneqq 0$. Then
	\begin{equation*}
		\hat{h}_{n_j} = \Delta_{n_j}\ind{A^c} + \ind{A^c} e^{Q_{n_j}\Delta_{n_j}}\hat{h}_{n_j}\,.
	\end{equation*}
	For $j$ large enough we eventually have $\Delta_{n_j}\xi<1$, and so by Proposition~\ref{prop:precise_uniform_limit}, we then have
	\begin{align*}
		\norm{\hat{h}_{n_j} - h_{Q_{n_j}}} &< \Delta_{n_j}L\norm{h_{Q_{n_j}}} \\
		&\leq \Delta_{n_j}L M\,,
	\end{align*} 
	with $L,M$ independent of $j$. Hence for $j$ large enough we have
	\begin{equation}\label{eq:prop:imprecise_limit:step_2}
		\norm{\hat{h}_{n_j} - h_{Q_{n_j}}} < \frac{\delta}{K}\,.
	\end{equation}
	
	Let next $\tilde{h}_{n_j}$ be the minimal non-negative solution to $\tilde{h}_{n_j}=\Delta_{n_j}\ind{A^c}+\ind{A^c} \Phi_{n_j}\tilde{h}_{n_j}$. Since $m_{n_j}\geq n_j$, for $j$ large enough we have $\norm{\Phi_{n_j}\resAc} < 1$ due to Assumption~\ref{ass:reachable}.
	
	By~\citep[Lemmas B.8 and B.12]{krak2021phd} we have
	\begin{equation*}
		\norm{\Phi_{n_j} - e^{Q_{n_j}\Delta_{n_j}}} \leq 2\Delta_{n_j}^2\norm{\rateset}^2\,,
	\end{equation*}
	and so, for any $\epsilon>0$, we can choose $j$ large enough so that eventually $\norm{\Phi_{n_j} - e^{Q_{n_j}\Delta_{n_j}}} <\epsilon$. Using the continuity of the map $T\mapsto(I-T)^{-1}$ on operators $T$ for which this inverse exists, for large enough $j$ we therefore find that
	\begin{align*}
		&\norm{\tilde{h}_{n_j}\resAc - \hat{h}_{n_j}\resAc}  \\
		&= \norm{\bigl((I-\Phi_{n_j}\resAc)^{-1} - (I-e^{Q_{n_j}\Delta_{n_j}}\resAc)^{-1}\bigr)\Delta_{n_j}\ones}  \\
		&< \Delta_{n_j}\frac{\delta}{K} \leq \frac{\delta}{K}\,. 
	\end{align*}
	Since $\tilde{h}_{n_j}\vert_A=0=\hat{h}_{n_j}\vert_A$, this implies that then also
	\begin{equation}\label{eq:prop:imprecise_limit:step_3}
		\norm{\tilde{h}_{n_j} - \hat{h}_{n_j}} < \frac{\delta}{K}\,. 
	\end{equation}
	
	Next, we recall that $\underline{h}_{\Delta_{n_j}} = h_{n_j}$, and
	\begin{equation*}
		\norm{T_{n_j} - \Phi_{n_j}} < \epsilon_{n_j}\,.
	\end{equation*} 
	Hence by continuity of the map $T\mapsto (I-T)^{-1}$ on operators $T$ for which this inverse exists, for large enough $j$ we find that
	\begin{align*}
		&\norm{h_{n_j}\resAc - \tilde{h}_{n_j}\resAc} \\
		&= \norm{\bigl((I-T_{n_j}\resAc)^{-1} - (I-\Phi_{n_j}\resAc)^{-1}\bigr)\Delta_{n_j}\ones} \\
		&< \Delta_{n_j}\frac{\delta}{K} \leq \frac{\delta}{K}\,.
	\end{align*}
	Since $h_{n_j}\vert_A = 0= \tilde{h}_{n_j}\vert_A$, this implies that also
	\begin{equation}\label{eq:prop:imprecise_limit:step_4}
		\norm{h_{n_j} - \tilde{h}_{n_j}} < \frac{\delta}{K}\,.
	\end{equation}
	
	Putting Equations~\eqref{eq:prop:imprecise_limit:step_1}--\eqref{eq:prop:imprecise_limit:step_4} together, we find that for any large enough $j$ it holds that
	\begin{align}
		\norm{\underline{h}_{\Delta_{n_j}} - h_*} &= \norm{h_{n_j} - h_*} \notag\\
		&\leq \norm{h_{n_j} - \tilde{h}_{n_j}} \notag\\
		&\quad + \norm{\tilde{h}_{n_j} - \hat{h}_{n_j}} \notag\\
		&\quad\quad\quad + \norm{\hat{h}_{n_j} - h_{Q_{n_j}}} \notag\\
		&\quad\quad\quad\quad + \norm{h_{Q_{n_j}} - h_*} \notag\\
		&< 4\frac{\delta}{K}\,. \label{eq:prop:imprecise_limit:jointsteps}
	\end{align}
	Since $\delta>0$ is arbitrary this clearly implies that
	\begin{equation}\label{eq:prop:imprecise_limit:sublimit}
		\lim_{j\to+\infty} \underline{h}_{\Delta_{n_j}} = h_*\,.
	\end{equation}
	Next, let us show that $h_*=\underline{h}$.	To this end, assume \emph{ex absurdo} that there is some $Q\in\rateset$ such that $h^Q(x) < h_*(x)$ for some $x\in A^c$. Let $\delta\coloneqq h_*(x) - h_Q(x) > 0$. Due to Corollary~\ref{cor:precise_discretisation_converges}, for any $\Delta>0$ small enough it holds that
	\begin{equation*}%\label{eq:prop:imprecise_limit:wrongstep}
		\norm{h_\Delta^Q - h^Q} < \frac{\delta}{K}\,.
	\end{equation*}
	This implies in particular that for large enough $j$ it holds that $h_{\Delta_{n_j}}^Q(x) < h^Q(x)+\nicefrac{\delta}{K}$. 
	Moreover, it follows from Equation~\eqref{eq:prop:imprecise_limit:jointsteps} that for large enough $j$ we have $\underline{h}_{\Delta_{n_j}}(x) > h_*(x) - 4\nicefrac{\delta}{K}$.
	It holds that $h_Q(x)=h_*(x)-\delta$, and hence, since $K=5$, we find that that for large enough $j$,
	\begin{align*}
		h_{\Delta_{n_j}}^Q(x) &< h^Q(x)+\nicefrac{\delta}{K} \\
		&= h_*(x)-\delta + \nicefrac{\delta}{K} \\
		&= h_*(x) - K\frac{\delta}{K} + \nicefrac{\delta}{K} \\
		&= h_*(x) - (K-1)\frac{\delta}{K} \\
		&= h_*(x) - 4\frac{\delta}{K} < \underline{h}_{\Delta_{n_j}}(x)\,.
	\end{align*}
	In other words, and using Lemma~\ref{lemma:discrete_infimum_bound}, we then have
	\begin{equation*}
		h^Q_{\Delta_{n_j}}(x) < \underline{h}_{\Delta_{n_j}}(x) = \inf_{T\in\mathcal{T}_{\Delta_{n_j}}}h_T(x) \leq h^Q_{\Delta_{n_j}}(x)\,,
	\end{equation*}
	where the last step used that $e^{Q\Delta_{n_j}}\in\mathcal{T}_{\Delta_{n_j}}$.
	From this contradiction we conclude that our earlier assumption must be wrong, and so it holds that $h_*(x)\leq h^Q(x)$ for all $x\in\states$ and $Q\in\rateset$. This implies that $h_*\leq \underline{h}$. Since it clearly also holds that $\underline{h}\leq h_*$ because $Q_*\in\rateset$, this implies that, indeed as claimed, $h_*=\underline{h}$.
	
	In summary, at this point we have shown that for any sequence $(\Delta_n)_{n\in\nats}$ in $\realspos$ with $\lim_{n\to+\infty}\Delta_n=0$, there is a subsequence such that $\lim_{j\to+\infty}\underline{h}_{\Delta_{n_j}}=\underline{h}$.
	
	So, finally, suppose \emph{ex absurdo} that $\lim_{\Delta\to 0^+}\underline{h}_{\Delta}\neq\underline{h}$. Then there is some sequence $(\Delta_n)_{n\in\nats}$ in $\realspos$ such that $\lim_{n\to+\infty}\Delta_n=0$, and some $\epsilon>0$, such that $\norm{\underline{h}_{\Delta_{n}}-\underline{h}}\geq\epsilon$ for all $n\in\nats$.
	By the above result, there is a subsequence such that $\lim_{j\to+\infty} \underline{h}_{\Delta_{n_j}}=\underline{h}$, which is a contradiction.
\end{proof}

\begin{proof}[Proof of Theorem~\ref{thm:hitting_times_invariant}]
	The crucial approach of this proof is to emulate~\citet[Sec 6.3]{erreygers2021phd} and consider \emph{discretized} and \emph{truncated} hitting times. By taking appropriate limits of such approximations, we then recover the ``real'' hitting times. We however need to be a bit careful with these constructions, since lower (and upper) expectation operators for continuous-time imprecise-Markov chains are not necessarily continuous with respect to arbitrary limits of such approximations~\cite[Chap~5]{erreygers2021phd}. This---fairly long---proof is therefore roughly divided into two parts; first, we construct a specific sequence of approximations, and establish the relevant continuity properties with respect to this sequence. Then, in the second part of this proof, we use this continuity to establish the main claim of this theorem.
	
	To this end, for any $t\in\realsnonneg$ and $\Delta\in\realspos$, we first consider a fixed-step grid $\nu_\Delta^t$ over $[0,t]$ with step-size $\Delta$, as
	\begin{equation}\label{eq:thm:hitting_times_invariant:grid_definition}
		\nu_\Delta^t \coloneqq \bigl\{i\Delta\,:\,i\in\natswith, i\Delta\leq t \bigr\}\,.
	\end{equation}
	We define the associated \emph{approximate} hitting time functions $\tau_\Delta^t:\Omega_{\realsnonneg}\to\reals$ for all $\omega\in\Omega_{\realsnonneg}$ as
	\begin{equation}\label{eq:thm:hitting_times_invariant:approximation_definition}
		\tau_\Delta^t(\omega) \coloneqq \min \Bigl(\bigl\{ s\in \nu_\Delta^t \,:\, \omega(s)\in A\}\cup\{t\}\Bigr)\,.
	\end{equation}
	Then by~\cite[Lemma 6.19]{erreygers2021phd}, as we take the time-horizon $t$ to infinity and the step-size $\Delta$ to zero, we have the point-wise limit to the actual hitting time function $\tau_{\realsnonneg}$, in that
	\begin{equation}\label{eq:thm:hitting_times_invariant:general_limit_approximations}
		\tau_{\realsnonneg}(\omega) = \lim_{t\to+\infty, \Delta\to 0^+}\tau_\Delta^t(\omega)\quad\text{for all $\omega\in\Omega_{\realsnonneg}$.}
	\end{equation}
	Let us now construct a specific sequence of approximate hitting time functions that will converge to this limit. To this end, first fix an arbitrary sequence $(\epsilon_n)_{n\in\natswith}$ in $\realspos$ such that $\lim_{n\to+\infty}\epsilon_n = 0$. Moreover, for any $n\in\natswith$, we introduce the (discrete-time) \emph{truncated} hitting time $\tau_{0:n}:\Omega_{\natswith}\to\reals$, defined for all $\omega\in\Omega_{\natswith}$ as
	\begin{equation*}
		\tau_{0:n}(\omega) \coloneqq \min\Bigl( \bigl\{t\in\{0,\ldots,n\}\,:\,\omega(t)\in A\bigr\}\cup\{n\} \Bigr)\,.
	\end{equation*}
	Now fix any $k\in\natswith$, let $\Delta_k\coloneqq 2^{-k}$, and let $\mathcal{T}_k$ denote the set of transition matrices that dominate $e^{\lrate\Delta_k}$. We now consider discrete-time imprecise-Markov chains parameterized by $\mathcal{T}_k$. As discussed in~\citep{krak2019hitting}, for all $n\in\natswith$ there are functions\footnote{These represent lower and an upper expectations with respect to a \emph{game-theoretic imprecise-Markov chain}, but the details don't concern us here.} $\underline{\mathbb{E}}_{\mathcal{T}_k}^{\mathrm{V}}[\tau_{0:n}\,\vert\,X_0]$ and $\overline{\mathbb{E}}_{\mathcal{T}_k}^{\mathrm{V}}[\tau_{0:n}\,\vert\,X_0]$ in $\gamblesX$ such that
	\begin{align*}
		\underline{\mathbb{E}}_{\mathcal{T}_k}^{\mathrm{V}}[\tau_{0:n}\,\vert\,X_0] &\leq \underline{\mathbb{E}}_{\mathcal{T}_k}^{\mathrm{I}}[\tau_{0:n}\,\vert\,X_0] \\
		&\leq \overline{\mathbb{E}}_{\mathcal{T}_k}^{\mathrm{I}}[\tau_{0:n}\,\vert\,X_0] \leq \overline{\mathbb{E}}_{\mathcal{T}_k}^{\mathrm{V}}[\tau_{0:n}\,\vert\,X_0]
	\end{align*}
	that, moreover, satisfy
	\begin{equation*}
		\lim_{n\to+\infty} \underline{\mathbb{E}}_{\mathcal{T}_k}^{\mathrm{V}}[\tau_{0:n}\,\vert\,X_0] = \underline{\mathbb{E}}_{\mathcal{T}_k}^{\mathrm{V}}[\tau_{\natswith}\,\vert\,X_0] = \underline{\mathbb{E}}_{\mathcal{T}_k}^{\mathrm{HM}}[\tau_{\natswith}\,\vert\,X_0]
	\end{equation*}
	and
	\begin{equation*}
		\lim_{n\to+\infty} \overline{\mathbb{E}}_{\mathcal{T}_k}^{\mathrm{V}}[\tau_{0:n}\,\vert\,X_0] = \overline{\mathbb{E}}_{\mathcal{T}_k}^{\mathrm{V}}[\tau_{\natswith}\,\vert\,X_0] = \overline{\mathbb{E}}_{\mathcal{T}_k}^{\mathrm{HM}}[\tau_{\natswith}\,\vert\,X_0]\,.
	\end{equation*}
	We already noted in Section~\ref{sec:hits_as_limits} that the functions $\underline{h}_{\Delta_k}$ and $\overline{h}_{\Delta_k}$ from Equations~\eqref{eq:lower_discretizes_system} and~\eqref{eq:upper_discretizes_system} satisfy
	\begin{equation*}
		\underline{h}_{\Delta_k} = \Delta_k\overline{\mathbb{E}}_{\mathcal{T}_k}^{\mathrm{HM}}[\tau_{\natswith}\,\vert\,X_0]
		\,\,\text{and}\,\,
		\overline{h}_{\Delta_k} = \Delta_k\overline{\mathbb{E}}_{\mathcal{T}_k}^{\mathrm{HM}}[\tau_{\natswith}\,\vert\,X_0]\,.
	\end{equation*}
	Combining the above, we find that
	\begin{equation*}
		\lim_{n\to+\infty} \Delta_k\underline{\mathbb{E}}_{\mathcal{T}_k}^{\mathrm{V}}[\tau_{0:n}\,\vert\,X_0] = \underline{h}_{\Delta_k}
	\end{equation*}
	and
	\begin{equation*}
		\lim_{n\to+\infty} \Delta_k\overline{\mathbb{E}}_{\mathcal{T}_k}^{\mathrm{V}}[\tau_{0:n}\,\vert\,X_0] = \overline{h}_{\Delta_k}\,.
	\end{equation*}
	Hence for all $k\in\natswith$, we can now choose $t_k\in\natswith$ large enough such that $t_k\geq k$, and so that with $n_k = 2^kt_k$ we have both
	\begin{equation}\label{eq:thm:hitting_times_invariant:lower_limit_discrete_by_game}
		\norm{\Delta_k\underline{\mathbb{E}}_{\mathcal{T}_k}^{\mathrm{V}}[\tau_{0:{n_k}}\,\vert\,X_0] - \underline{h}_{\Delta_k}} < \epsilon_k\,,
	\end{equation}
	and
	\begin{equation}\label{eq:thm:hitting_times_invariant:upper_limit_discrete_by_game}
		\norm{\Delta_k\overline{\mathbb{E}}_{\mathcal{T}_k}^{\mathrm{V}}[\tau_{0:{n_k}}\,\vert\,X_0] - \overline{h}_{\Delta_k}} < \epsilon_k\,.
	\end{equation}
	With these selections, we now define the sequence $(\tau_k)_{k\in\natswith}$ of approximate hitting times as $\tau_k\coloneqq \tau_{\Delta_k}^{t_k}$ for all $k\in\natswith$. Clearly we have $\lim_{k\to+\infty}\Delta_k=0$, and since $t_k\geq k$ we also find that $\lim_{k\to+\infty}t_k=+\infty$. Hence by Equation~\eqref{eq:thm:hitting_times_invariant:general_limit_approximations} we have the pointwise limit
	\begin{equation}\label{eq:thm:hitting_times_invariant:concrete_limit_approximations}
		\tau_{\realsnonneg}(\omega) = \lim_{k\to+\infty} \tau_k(\omega)\quad\text{for all $\omega\in\Omega_{\realsnonneg}$.}
	\end{equation}
	Having constructed this specific sequence that converges to the ``true'' hitting time function, we will now demonstrate the relevant continuity properties of the lower- and upper expectations of interest, with respect to this sequence.
	
	To this end, we define $\hat{\tau}:\Omega_{\realsnonneg}\to\reals\cup\{+\infty\}$ as
	\begin{equation}\label{eq:thm:hitting_times_invariant:bounder_definition}
		\hat{\tau}(\omega) \coloneqq \sup_{t\in\natswith}\sup_{n\in\natswith} \tau_{\Delta_n}^t(\omega)
		\quad\text{for all $\omega\in\Omega_{\realsnonneg}$.}
	\end{equation}
	Then for all $k\in\natswith$ we have $\tau_k(\omega)=\tau_{\Delta_k}^{t_k}(\omega)\leq \hat{\tau}(\omega)$ for all $\omega\in\Omega_{\realsnonneg}$. Moreover, since every $\tau_k$ is non-negative, it holds in fact that $\abs{\tau_k(\omega)}\leq \hat{\tau}(\omega)$ for all $k\in\natswith$ and $\omega\in\Omega_{\realsnonneg}$. This means that if we can show that the upper expectation $\smash{\overline{\mathbb{E}}_\rateset^\mathrm{I}}[\hat{\tau}\,\vert\,X_0=x]$ is bounded for all $x\in\states$, then we can use the imprecise version of the dominated convergence theorem~\cite[Thm 5.32]{erreygers2021phd} to take lower- and upper expectations of the limit in Equation~\eqref{eq:thm:hitting_times_invariant:concrete_limit_approximations}. So, we will now show that this boundedness indeed holds.

	We note that, for fixed $t\in\natswith$, $\tau_{\Delta_n}^t$ is monotonically \emph{decreasing} as we increase $n\in\natswith$. To see this, first consider the grids $\nu_{\Delta_{n}}^t$ and $\nu_{\Delta_{n+1}}^t$ over $[0,t]$. For any $s\in\nu_{\Delta_n}^t$ there is some $i\in\natswith$ such that $s=i\Delta_n$, and since $\Delta_n=2^{-n}=2\Delta_{n+1}$, we find that also $s=2i\Delta_{n+1}\in\nu_{\Delta_{n+1}}^t$. Hence we conclude that $\nu_{\Delta_n}^t \subseteq \nu_{\Delta_{n+1}}^t$. From this set inclusion, we also clearly have for any $\omega\in\Omega_{\realsnonneg}$ that
	\begin{equation*}
		\{ s\in\nu_{\Delta_n}^t\,:\,\omega(s)\in A \} \subseteq \{ s\in\nu_{\Delta_{n+1}}^t\,:\,\omega(s)\in A \}\,,
	\end{equation*}
	and so together with the fact that $s\leq t$ for all $s\in \nu_{\Delta_{n+1}}^t$, it then follows from Equation~\eqref{eq:thm:hitting_times_invariant:approximation_definition} that $\tau_{\Delta_n}^t(\omega) \geq \tau_{\Delta_{n+1}}^t(\omega)$.

	Using this observation, we immediately find that for any $t\in\natswith$ and $\omega\in\Omega_{\realsnonneg}$ it holds that
	\begin{equation*}
		\sup_{n\in\natswith} \tau_{\Delta_n}^t(\omega) = \tau_{\Delta_0}^t(\omega) = \tau_{1}^t(\omega)\,,
	\end{equation*}
	and so from Equation~\eqref{eq:thm:hitting_times_invariant:bounder_definition}, we have
	\begin{equation*}
		\hat{\tau}(\omega) = \sup_{t\in\natswith}\tau_{1}^t(\omega)\,.
	\end{equation*}
	Next, we observe that $\tau_1^t$ is monotonically increasing as we increase $t\in\natswith$. Indeed, for any $t\in\natswith$ the grid $\nu_1^t$ over $[0,t]$ simply constitutes the set $\nu_1^t=\{0,\ldots,t\}$. Hence that $\tau_1^t$ is monotonically increasing as we increase $t\in\natswith$, follows immediately from Equation~\eqref{eq:thm:hitting_times_invariant:approximation_definition}. In particular, this implies that the sequence $(\tau_1^t)_{t\in\natswith}$ converges \emph{monotonically} to $\hat{\tau}$. Moreover, we have $\tau_1^0(\omega)=0$ for all $\omega\in\Omega_{\realsnonneg}$, and so we find that identically 
	\begin{align*}
		\underline{\mathbb{E}}_\rateset^{\mathrm{I}}[\tau_1^0\,\vert\,X_0] = \inf_{P\in\mathcal{P}_\rateset^\mathrm{I}}\mathbb{E}_P[\tau_1^0\,\vert\,X_0] = 0\,.
	\end{align*} 
	Hence by the continuity of upper expectations with respect to monotonically increasing convergent sequences of functions that are bounded below~\cite[Thm 5.31]{erreygers2021phd}, we have
	\begin{equation}\label{eq:thm:hitting_times_invariant:limit_bounder}
		\overline{\mathbb{E}}_\rateset^\mathrm{I}[\hat{\tau}\,\vert\,X_0] = \lim_{t\to+\infty, t\in\natswith}\overline{\mathbb{E}}_\rateset^\mathrm{I}[\tau_1^t\,\vert\,X_0]\,.
	\end{equation}

	Now, for every $t\in\natswith$, $\tau_{1}^t$ only depends on finitely many time-points; indeed, $\tau_{1}^t(\omega)$ only depends on the value of $\omega(s)$ for $s\in\{0,\ldots,t\}$. Using~\cite[Thm 7.2]{krak2021phd}, this means that the lower- and upper expectations of these functions with respect to the imprecise-Markov chain $\mathcal{P}^\mathrm{I}_{\rateset}$, can also be expressed as lower (resp. upper) expectations of this function with respect to an induced \emph{discrete}-time imprecise-Markov chain. Indeed, since the step-size used in these approximating functions is uniformly equal to one, and using the obvious correspondence between $\tau_1^t$ and $\tau_{0:t}$, it is not difficult to see that
	\begin{equation}\label{eq:thm:hitting_times_invariant:cont_disc_equivalence_sequence_bounders}
		\overline{\mathbb{E}}_\rateset^\mathrm{I}[\tau_1^t\,\vert\,X_0] = \overline{\mathbb{E}}_{\mathcal{T}}^\mathrm{I}[\tau_{0:t}\,\vert\,X_0]\quad\text{for all $t\in\natswith$,}
	\end{equation}
	where $\mathcal{T}$ is the set of transition matrices that dominates $e^{\lrate}$.
	
	We now again invoke the previously mentioned results from~\citep{krak2019hitting}; for any $t\in\natswith$, there is a function $\overline{\mathbb{E}}_{\mathcal{T}}^\mathrm{V}[\tau_{0:t}\,\vert\,X_0]\in\gamblesX$ that satisfies
	\begin{equation}\label{eq:thm:hitting_times_invariant:upper_bound_gametheoretic}
		\overline{\mathbb{E}}_{\mathcal{T}}^\mathrm{I}[\tau_{0:t}\,\vert\,X_0] \leq \overline{\mathbb{E}}_{\mathcal{T}}^\mathrm{V}[\tau_{0:t}\,\vert\,X_0]\,.
	\end{equation}
	Combining Equations~\eqref{eq:thm:hitting_times_invariant:limit_bounder}, \eqref{eq:thm:hitting_times_invariant:cont_disc_equivalence_sequence_bounders}, and~\eqref{eq:thm:hitting_times_invariant:upper_bound_gametheoretic}, and using~\cite[Prop 7]{krak2019hitting} to establish the limit on the final right-hand side, we have
	\begin{align*}
		\overline{\mathbb{E}}_\rateset^\mathrm{I}[\hat{\tau}\,\vert\,X_0] &= \lim_{t\to+\infty, t\in\natswith}\overline{\mathbb{E}}_\rateset^\mathrm{I}[\tau_1^t\,\vert\,X_0] \\
		&= \lim_{t\to+\infty, t\in\natswith}  \overline{\mathbb{E}}_{\mathcal{T}}^\mathrm{I}[\tau_{0:t}\,\vert\,X_0] \\
		&\leq \lim_{t\to+\infty, t\in\natswith}  \overline{\mathbb{E}}_{\mathcal{T}}^\mathrm{V}[\tau_{0:t}\,\vert\,X_0] =  \overline{\mathbb{E}}_{\mathcal{T}}^\mathrm{V}[\tau_{\natswith}\,\vert\,X_0] \,.
	\end{align*}
	By~\cite[Thm 12]{krak2019hitting} it holds that
	\begin{equation*}
		\overline{\mathbb{E}}_{\mathcal{T}}^\mathrm{V}[\tau_{\natswith}\,\vert\,X_0] = \overline{\mathbb{E}}_{\mathcal{T}}^\mathrm{HM}[\tau_{\natswith}\,\vert\,X_0]\,,
	\end{equation*}
	and, moreover, that there is some homogeneous discrete-time Markov chain $P\in\mathcal{P}^\mathrm{HM}_\mathcal{T}$ with associated transition matrix $T={}^PT\in\mathcal{T}$ and hitting times $h=\mathbb{E}_P[\tau_{\natswith}\,\vert\,X_0]$ such that $h = \overline{\mathbb{E}}_{\mathcal{T}}^\mathrm{HM}[\tau_{\natswith}\,\vert\,X_0]$. Putting this together, we find that
	\begin{equation}\label{eq:thm:hitting_times_invariant:bounded_upper_bound_by_homogen}
		\overline{\mathbb{E}}_\rateset^\mathrm{I}[\hat{\tau}\,\vert\,X_0=x] \leq h(x)\quad\text{for all $x\in\states$.}
	\end{equation}
	By Proposition~\ref{prop:precise_discr_system}, $h$ is also the minimal non-negative solution to the system
	\begin{equation}\label{eq:thm:hitting_times_invariant:system_reaching_upper_discrete}
		h = \ind{A^c} + \ind{A^c} T h\,.
	\end{equation}
	It is immediate from the definition that $h\vert_A=0$, and since $\hat{\tau}$ is clearly non-negative, we obtain from Equation~\eqref{eq:thm:hitting_times_invariant:bounded_upper_bound_by_homogen} that for all $x\in A$ we have
	\begin{equation*}
		0 \leq \overline{\mathbb{E}}_\rateset^\mathrm{I}[\hat{\tau}\,\vert\,X_0=x] \leq h(x) = 0\,,
	\end{equation*}
	or in other words, that $\overline{\mathbb{E}}_\rateset^\mathrm{I}[\hat{\tau}\,\vert\,X_0=x]=0$ for all $x\in A$. So, it remains to bound this upper expectation on $A^c$.
	
	By our Assumption~\ref{ass:reachable}, it holds for all $x\in A^c$ that $e^{\lrate}\ind{A}(x) >0$. Since $e^{\lrate}$ is the lower transition operator corresponding to $\mathcal{T}$ due to Proposition~\ref{prop:duality_trans}, it follows that $e^{\lrate}$ satisfies conditions C1--C3 and R1 from Reference~\citep{krak2020computing}.
	We now recall that $T={}^PT\in\mathcal{T}$. Since the preconditions C1--C3 and R1 of this reference are all satisfied, we can now invoke~\cite[Lemma 10]{krak2020computing}, which states that the inverse operator $(I-T\resAc)^{-1}$ exists.
	
	We note that $h\vert_A=0$, and so $h=(h\resAc)\upX$. Hence in particular, we have $T\resAc h\resAc = (Th)\resAc$. From Equation~\eqref{eq:thm:hitting_times_invariant:system_reaching_upper_discrete}, we now find that
	\begin{equation*}
		h\resAc = \ones + (Th)\resAc = \ones + T\resAc h\resAc\,,
	\end{equation*}
	and so re-ordering terms, we have $(I-T\resAc)h\resAc = \ones$. Using the existence of the inverse operator established above, we obtain
	\begin{equation*}
		h\resAc = (I-T\resAc)^{-1}\ones\,.
	\end{equation*}
	Since $(I-T\resAc)$ is an invertible bounded linear operator, also clearly $(I-T\resAc)^{-1}$ is bounded. Hence we have
	\begin{equation*}
		\norm{h\resAc} = \norm{(I-T\resAc)^{-1}\ones} \leq \norm{(I-T\resAc)^{-1}} < +\infty\,.
	\end{equation*}
	From Equation~\eqref{eq:thm:hitting_times_invariant:bounded_upper_bound_by_homogen} we find that $\overline{\mathbb{E}}_\rateset^\mathrm{I}[\hat{\tau}\,\vert\,X_0=x]<+\infty$ for all $x\in A^c$. In summary, at this point we have shown that $\overline{\mathbb{E}}_\rateset^\mathrm{I}[\hat{\tau}\,\vert\,X_0=x]$ is bounded for all $x\in\states$. Since we already established that $\hat{\tau}$ absolutely dominates the sequence $(\tau_k)_{k\in\natswith}$, we can now finally use the limit~\eqref{eq:thm:hitting_times_invariant:concrete_limit_approximations} and the dominated convergence theorem~\cite[Thm 5.32]{erreygers2021phd} to establish that
	\begin{equation}\label{eq:thm:hitting_times_invariant:limsup_bound_on_target}
		\limsup_{k\to+\infty} \underline{\mathbb{E}}^\mathrm{I}_\rateset[\tau_k\,\vert\,X_0] \leq \underline{\mathbb{E}}^\mathrm{I}_\rateset[\tau_{\realsnonneg}\,\vert\,X_0]\,,
	\end{equation}
	and
	\begin{equation}\label{eq:thm:hitting_times_invariant:liminf_bound_on_target}
		\overline{\mathbb{E}}^\mathrm{I}_\rateset[\tau_{\realsnonneg}\,\vert\,X_0] \leq
		\liminf_{k\to+\infty} \overline{\mathbb{E}}^\mathrm{I}_\rateset[\tau_k\,\vert\,X_0]\,.
	\end{equation}
	This concludes the first part of this proof. Our next step will be to identify the limits superior and inferior in the above inequalities as corresponding to, respectively, $\underline{h}$ and $\overline{h}$.
	
	Let us start by obtaining the required result for the lower expectation. From the definition of the limit superior, there is a convergent subsequence such that
	\begin{equation}\label{eq:thm:hitting_times_invariant:limsup_limit}
		\underline{s}\coloneqq \lim_{j\to+\infty}\underline{\mathbb{E}}^\mathrm{I}_\rateset[\tau_{k_j}\,\vert\,X_0] = \limsup_{k\to+\infty} \underline{\mathbb{E}}^\mathrm{I}_\rateset[\tau_k\,\vert\,X_0]\,.
	\end{equation}
	Now fix any $j\in\natswith$, and consider the approximate function $\tau_{k_j}=\tau_{\smash{\Delta_{k_j}}}^{t_{k_j}}$. As before, this function really only depends on the system at finitely many time points; specifically, those on the grid $\nu_{\smash{\Delta_{k_j}}}^{t_{k_j}}$ over $[0,t_{k_j}]$. We can therefore again use~\cite[Thm 7.2]{krak2021phd}, to express the lower- and upper expectations of this function with respect to the imprecise-Markov chain $\mathcal{P}^\mathrm{I}_{\rateset}$, as lower- and upper expectations of a function with respect to an induced \emph{discrete}-time imprecise-Markov chain. Since the step size of this grid is now equal to $\Delta_{k_j}$ rather than one, this requires a bit more effort than before. In particular, we now need to compensate for the step-size $\Delta_{k_j}$ of the grid. Indeed, the corresponding discrete-time imprecise-Markov chain should consider steps that are implicitly of this ``length'', so we consider the model induced by the set $\mathcal{T}_{k_j}$ of transition matrices that dominate $e^{\lrate\Delta_{k_j}}$. It then remains to find an appropriate translation $\tilde{\tau}_{k_j}$ of $\tau_{k_j}$ to the domain $\Omega_{\natswith}$. 
	
	As a first observation, we note that this ``translation'' $\tilde{\tau}_{k_j}:\Omega_{\natswith}\to\reals$ should depend on the same number of time points as $\tau_{k_j}$. We note that since $t_k\in\natswith$ it holds that $t_k\in\nu^{t_{k_j}}_{\smash{\Delta_{k_j}}}$. Hence it follows from Equation~\eqref{eq:thm:hitting_times_invariant:grid_definition} that $\nu^{t_{k_j}}_{\smash{\Delta_{k_j}}}$ contains exactly $\Delta_{k_j}^{-1}t_{k_j}=2^{k_j}t_{k_j}=n_{k_j}$ time points, in addition to the origin $0$, and that $\tau_{k_j}$ depends exactly on these time points. Indeed, inspection of Equation~\eqref{eq:thm:hitting_times_invariant:approximation_definition} reveals that, by re-scaling to compensate for the step size $\Delta_{k_j}$, the quantity $\Delta_{k_j}^{-1}\tau_{k_j}(\omega)$ simply represents the natural index of the discrete grid element of $\nu^{t_{k_j}}_{\smash{\Delta_{k_j}}}$ on which $\omega\in\Omega_{\realsnonneg}$ did (or did not) initially hit $A$.   
	Adapting Equation~\eqref{eq:thm:hitting_times_invariant:approximation_definition}, we therefore define for any $\omega\in\Omega_{\natswith}$ that
	\begin{align*}
		\tilde{\tau}_{k_j}(\omega) \coloneqq \min\Bigl( \bigl\{ s\in\nu_{\Delta_{k_j}}^{t_{k_j}}\,:\,\omega\bigl(\Delta_{k_j}^{-1}s\bigr)\in A\bigr\} \cup \{t_{k_j}\}\Bigr)\,.
	\end{align*}
	We see that, as required, $\Delta_{k_j}^{-1}\tilde{\tau}_{k_j}(\omega)$ is again simply the identity of the step on which $\omega\in\Omega_{\natswith}$ did (or did not) initially hit $A$. This implies the relation to the discrete-time truncated hitting time $\tau_{0:n_{k_j}}$; for any $\omega\in\Omega_{\natswith}$ we have
	\begin{align*}
		&\tilde{\tau}_{k_j}(\omega) \\
		&= \Delta_{k_j}\Delta_{k_j}^{-1}\tilde{\tau}_{k_j}(\omega) \\
		&= \Delta_{k_j}\min\Bigl( \bigl\{ s\in\Delta_{k_j}^{-1}\nu_{\Delta_{k_j}}^{t_{k_j}}\,:\,\omega(s)\in A\bigr\} \cup \{\Delta_{k_j}^{-1}t_{k_j}\}\Bigr) \\
		&= \Delta_{k_j}\min\Bigl( \bigl\{ s\in\{0,\ldots,n_{k_j}\}\,:\,\omega\bigl(s\bigr)\in A\bigr\} \cup \{n_{k_j}\}\Bigr) \\
		&= \Delta_{k_j}\tau_{0:n_{k_j}}(\omega)\,,
	\end{align*}
	and so we simply have that $\tilde{\tau}_{k_j}=\Delta_{k_j}\tau_{0:n_{k_j}}$. Following the discussion in~\cite[Chap 7]{krak2021phd}, and~\cite[Thm 7.2]{krak2021phd} in particular, we therefore find the identity
	\begin{equation}\label{eq:thm:hitting_times_invariant:discretised_representation}
		\underline{\mathbb{E}}^\mathrm{I}_\rateset[\tau_{k_j}\,\vert\,X_0] = \Delta_{k_j}\underline{\mathbb{E}}^\mathrm{I}_{\mathcal{T}_{k_j}}[\tau_{0:n_{k_j}}\,\vert\,X_0]\quad\text{for all $j\in\natswith$.}
	\end{equation}
	We again recall from~\citet{krak2019hitting} the objects $\underline{\mathbb{E}}^\mathrm{V}_{\mathcal{T}_{k_j}}[\tau_{0:n_{k_j}}\,\vert\,X_0]$ in $\gamblesX$ satisfying
	\begin{equation*}
		\underline{\mathbb{E}}^\mathrm{V}_{\mathcal{T}_{k_j}}[\tau_{0:n_{k_j}}\,\vert\,X_0] \leq \underline{\mathbb{E}}^\mathrm{I}_{\mathcal{T}_{k_j}}[\tau_{0:n_{k_j}}\,\vert\,X_0]\quad\text{for all $j\in\natswith$.}
	\end{equation*}
	Hence from Equations~\eqref{eq:thm:hitting_times_invariant:limsup_limit} and~\eqref{eq:thm:hitting_times_invariant:discretised_representation} we now find that
	\begin{align}
		\underline{s} &= \lim_{j\to+\infty} \underline{\mathbb{E}}^\mathrm{I}_\rateset[\tau_{k_j}\,\vert\,X_0] \notag\\
		&= \lim_{j\to+\infty} \Delta_{k_j}\underline{\mathbb{E}}^\mathrm{I}_{\mathcal{T}_{k_j}}[\tau_{0:n_{k_j}}\,\vert\,X_0] \notag\\
		&\geq \lim_{j\to+\infty} \Delta_{k_j}\underline{\mathbb{E}}^\mathrm{V}_{\mathcal{T}_{k_j}}[\tau_{0:n_{k_j}}\,\vert\,X_0]\,.\label{eq:thm:hitting_times_invariant:limsup_lower_bound_homogen}
	\end{align}
	We now note that, for all $j\in\natswith$, we have
	\begin{align*}
		&\norm{\Delta_{k_j}\underline{\mathbb{E}}^\mathrm{V}_{\mathcal{T}_{k_j}}[\tau_{0:n_{k_j}}\,\vert\,X_0] - \underline{h}} \\
		&\quad\leq \norm{\Delta_{k_j}\underline{\mathbb{E}}^\mathrm{V}_{\mathcal{T}_{k_j}}[\tau_{0:n_{k_j}}\,\vert\,X_0] - \underline{h}_{\Delta_{k_j}}} + \norm{\underline{h}_{\Delta_{k_j}} - \underline{h}} \\
		&\quad< \epsilon_{k_j} + \norm{\underline{h}_{\Delta_{k_j}} - \underline{h}}\,,
	\end{align*}
	where we used Equation~\eqref{eq:thm:hitting_times_invariant:lower_limit_discrete_by_game} for the final inequality. Using that $\lim_{j\to+\infty}\epsilon_{k_j}=0$ and $\lim_{j\to+\infty}\Delta_{k_j}=0$, together with Proposition~\ref{prop:imprecise_limit}, we see that both summands vanish as we increase $j\in\natswith$, and so we have
	\begin{equation}\label{eq:thm:hitting_times_invariant:limit_to_homogen}
		\lim_{j\to+\infty} \Delta_{k_j}\underline{\mathbb{E}}^\mathrm{V}_{\mathcal{T}_{k_j}}[\tau_{0:n_{k_j}}\,\vert\,X_0] = \underline{h}\,.
	\end{equation}
	We already established in Section~\ref{sec:hits_as_limits} that $\underline{h}=\underline{\mathbb{E}}^\mathrm{HM}_\rateset[\tau_{\realsnonneg}\,\vert\,X_0]$. Hence by combining Equations~\eqref{eq:thm:hitting_times_invariant:limsup_bound_on_target}, \eqref{eq:thm:hitting_times_invariant:limsup_limit}, \eqref{eq:thm:hitting_times_invariant:limsup_lower_bound_homogen}, and~\eqref{eq:thm:hitting_times_invariant:limit_to_homogen}, we now find that
	\begin{equation}\label{eq:thm:hitting_times_invariant:reverse_bound}
		\underline{\mathbb{E}}^\mathrm{HM}_\rateset[\tau_{\realsnonneg}\,\vert\,X_0] = \underline{h} \leq \underline{s} \leq \underline{\mathbb{E}}^\mathrm{I}_\rateset[\tau_{\realsnonneg}\,\vert\,X_0]\,.
	\end{equation}
	However, as noted in Section~\ref{subsec:imc_sets_types} we have the inclusion $\mathcal{P}^\mathrm{HM}_\rateset\subseteq \mathcal{P}^\mathrm{M}_\rateset\subseteq \mathcal{P}^\mathrm{I}_{\rateset}$, so it immediately follows from the definition of the lower expectations that
	\begin{equation*}
		\underline{\mathbb{E}}^\mathrm{I}_\rateset[\tau_{\realsnonneg}\,\vert\,X_0] \leq \underline{\mathbb{E}}^\mathrm{M}_\rateset[\tau_{\realsnonneg}\,\vert\,X_0] \leq \underline{\mathbb{E}}^\mathrm{HM}_\rateset[\tau_{\realsnonneg}\,\vert\,X_0]\,.
	\end{equation*}
	Hence by Equation~\eqref{eq:thm:hitting_times_invariant:reverse_bound} we obtain the identity
	\begin{equation*}
		\underline{\mathbb{E}}^\mathrm{I}_\rateset[\tau_{\realsnonneg}\,\vert\,X_0] = \underline{\mathbb{E}}^\mathrm{M}_\rateset[\tau_{\realsnonneg}\,\vert\,X_0] = \underline{\mathbb{E}}^\mathrm{HM}_\rateset[\tau_{\realsnonneg}\,\vert\,X_0]\,,
	\end{equation*}
	which concludes the proof that the lower expected hitting times are the same for all three types of continuous-time imprecise-Markov chains. We omit the proof for the upper expected hitting times; this is completely analogous, starting instead from Equation~\eqref{eq:thm:hitting_times_invariant:liminf_bound_on_target} and using the norm bound~\eqref{eq:thm:hitting_times_invariant:upper_limit_discrete_by_game} to pass to the limit $\overline{h}=\overline{\mathbb{E}}^\mathrm{HM}_\rateset[\tau_{\realsnonneg}\,\vert\,X_0]$.
\end{proof}

\begin{proof}[Proof of Theorem~\ref{thm:lower_upper_hitting_system}]
	First fix any $\Delta>0$. For $\underline{h}_\Delta$ it then holds that
	\begin{align*}
		\underline{h}_{\Delta} = \Delta\ind{A^c}+\ind{A^c} e^{\lrate \Delta}\underline{h}_{\Delta}\,.
	\end{align*}
	Since $\underline{h}_\Delta(x)=0$ for all $x\in A$, we have $\underline{h}_\Delta=\ind{A^c} \underline{h}_\Delta$ and $\ind{A}\underline{h}_\Delta=0$. We can therefore rearrange terms and add $\ind{A}\underline{h}_\Delta$ to the above, to obtain
	\begin{equation*}
		\ind{A}\underline{h}_\Delta = \ind{A^c} + \ind{A^c}\frac{e^{\lrate\Delta} - I}{\Delta}\underline{h}_{\Delta}\,.
	\end{equation*}
	Because the individual limits exist by Proposition~\ref{prop:imprecise_limit} and~\citep[Prop 9]{de2017limit}, taking $\Delta$ to zero yields 
	\begin{align*}
		\ind{A}\underline{h} &= \ind{A}\lim_{\Delta\to 0^+} \underline{h}_\Delta \\
		&= \ind{A^c} + \ind{A^c}\lim_{\Delta\to 0^+}\frac{e^{\lrate\Delta} - I}{\Delta}\underline{h}_{\Delta} \\
		&= \ind{A^c} + \ind{A^c}\lim_{\Delta\to 0^+}\frac{e^{\lrate\Delta} - I}{\Delta}\lim_{\Delta\to 0^+}\underline{h}_\Delta \\
		&= \ind{A^c} + \ind{A^c}\lrate \,\underline{\vphantom{Q}h}\,.
	\end{align*}
	So, $\underline{h}$ is indeed a solution to the system $\ind{A}\underline{h}=\ind{A^c}+\ind{A^c}\lrate\,\underline{\vphantom{Q}h}$. It follows from a completely analogous argument that also $\overline{h}$ is a solution to $\ind{A}\overline{h}=\ind{A^c}+\ind{A^c}\urate\,\overline{h}$.
	
	That $\underline{h}$ and $\overline{h}$ are non-negative is clear. We now first show that $\underline{h}$ is the minimal non-negative solution to its corresponding system. 
	To this end, suppose \emph{ex absurdo} that there is some non-negative $h\in\gamblesX$ such that $h(x)<\underline{h}(x)$ for some $x\in\states$, and $\ind{A}h=\ind{A^c}+\ind{A^c}\lrate h$. Then clearly $x\in A^c$ since $\underline{h}(y)=0$ for all $y\in A$ and $h$ is non-negative.
	
	By Proposition~\ref{prop:duality_rate}, there is then some $Q\in\rateset$ such that $Qh=\lrate h$, and so also $\ind{A}h=\ind{A^c}+\ind{A^c}Qh$. By Proposition~\ref{prop:precise_cont_system} there is some minimal non-negative solution $h_*$ to the system $\ind{A}h_*=\ind{A^c}+\ind{A^c}Qh_*$, where the minimality implies in particular that $h_*\leq h$. Since $Q\in\rateset$, we obtain
	\begin{equation*}
		h_*(x) \leq h(x) < \underline{h}(x) = \inf_{Q'\in\rateset} h^{Q'}(x) \leq h_*(x)\,,
	\end{equation*}
	which is a contradiction.
	
	We next show that $\overline{h}$ is the minimal non-negative solution to its corresponding system; this will require a bit more effort and we need to start with some auxiliary constructions. By Proposition~\ref{prop:duality_rate}, there is some $Q\in\rateset$ such that $Q\overline{h}=\urate\, \overline{h}$. Let $G$ be the subgenerator of $Q$. 
	
	Consider the $\xi>0$ from Proposition~\ref{prop:subsemigroup_ues}, and let $\norm{\cdot}_*$ be the norm from Section~\ref{sec:quasicontractive}. Since $\gamblesAc$ is finite-dimensional, the norms $\norm{\cdot}$ and $\norm{\cdot}_*$ are equivalent, whence there is some $c\in\realspos$ such that $\norm{f}_*\leq c\norm{f}$ for all $f\in\gamblesAc$. 
	
	Now let $\Delta>0$ be such that $\Delta\norm{Q}\leq 1$, $\Delta\norm{\lrate}\leq 1$, $\Delta\xi<\nicefrac{2}{3}$, and $\Delta c\norm{Q}^2<\nicefrac{\xi}{3}$; this is clearly always possible. Define the map $H:\gamblesAc\to\gamblesAc$ for all $f\in\gamblesAc$ as
	\begin{equation*}
		H(f) \coloneqq f + \Delta\mathbf{1} + \Delta Gf = \Delta\ones + (I+\Delta G)f\,.
	\end{equation*}
	Let us show that $H$ is a contraction on the Banach space $(\gamblesAc, \norm{\cdot}_*)$, or in other words, that there is some $\alpha\in[0,1)$ such that $\norm{H(f)-H(g)}_*\leq \alpha\norm{f-g}_*$ for all $f,g\in\gamblesAc$~\cite[Sec 10.1.1]{renardyrogers2004intropde}. So, fix any $f,g\in\gamblesAc$. Then we have
	\begin{align*}
		\norm{H(f) - H(g)}_* &= \norm{(I+\Delta G)f - (I+\Delta G)g}_* \\
		&= \norm{(I+\Delta G)(f-g)}_* \\
		&\leq \norm{I+\Delta G}_*\norm{f-g}_*\,,
	\end{align*}
	from which we find that
	\begin{align}
		&\norm{H(f) - H(g)}_* \notag\\
		&\quad\quad\leq \left(\norm{(I+\Delta G) - e^{G\Delta}}_* + \norm{e^{G\Delta}}_*\right)\norm{f-g}_*\,.\label{eq:thm:lower_upper_hitting_system:bound_H_contract}
	\end{align}
	By Proposition~\ref{prop:precise_quasicontractive} we have $\norm{e^{G\Delta}}_*\leq e^{-\xi\Delta}$, and so using a standard quadratic bound on the negative scalar exponential,
	\begin{equation}
		\norm{e^{G\Delta}}_* \leq 1 - \Delta\xi + \frac{1}{2}\Delta^2\xi^2 < 1 - \frac{2}{3}\Delta\xi\,,\label{eq:thm:lower_upper_hitting_system:bound_H_quasicontract}
	\end{equation} 
	where we used that $\Delta\xi < \nicefrac{2}{3}$.
	
	Moreover, we have that
	\begin{align*}
		\norm{(I+\Delta G) - e^{G\Delta}}_* &\leq c\norm{(I+\Delta G) - e^{G\Delta}} \\
		&\leq c\norm{(I+\Delta Q) - e^{Q\Delta}} \\
		&\leq c\Delta^2\norm{Q}^2\,,
	\end{align*}
	where the second inequality used Lemmas~\ref{lemma:linear_approx_identity_on_A} and~\ref{lemma:restriction_norm_bound} and Corollary~\ref{cor:exponential_identity_on_A}; and the final inequality used~\cite[Lemma B.8]{krak2021phd}. Since $\Delta c\norm{Q}^2<\nicefrac{\xi}{3}$ we have
	\begin{equation}
		\norm{(I+\Delta G) - e^{G\Delta}}_* < \frac{1}{3}\Delta\xi\,.\label{eq:thm:lower_upper_hitting_system:bound_H_contract_trans}
	\end{equation}
	Combining Equations~\eqref{eq:thm:lower_upper_hitting_system:bound_H_contract},~\eqref{eq:thm:lower_upper_hitting_system:bound_H_quasicontract}, and~\eqref{eq:thm:lower_upper_hitting_system:bound_H_contract_trans} we obtain
	\begin{equation*}
		\norm{H(f) - H(g)}_* \leq \left(1-\frac{1}{3}\Delta\xi\right)\norm{f-g}_*\,.
	\end{equation*}
	Since $\Delta>0$, $\xi>0$, and $\Delta\xi<\nicefrac{2}{3}$, we conclude that $H$ is indeed a contraction. Hence by the Banach fixed-point theorem~\cite[Thm 10.1]{renardyrogers2004intropde}, there is a unique fixed point $f\in\gamblesAc$ such that $H(f)=f$ and, for any $g\in\gamblesAc$, it holds that
	\begin{equation}\label{eq:thm:lower_upper_hitting_system:abstract_banach_limit}
		\lim_{n\to+\infty} H^n(g) = f\,.
	\end{equation}
	It is easy to see that this unique fixed point is given by $\overline{h}\resAc$. Indeed, from the choice of $Q$ and the fact that $\overline{h}$ satisfies $\ind{A}\overline{h}=\ind{A^c}+\ind{A^c}\urate\,\overline{h}$, we have 
	\begin{equation*}
		\ind{A}\overline{h}=\ind{A^c}+\ind{A^c} Q\overline{h} =\ind{A^c}+ \ind{A^c}\urate\,\overline{h}\,.
	\end{equation*}
	Moreover, since $\overline{h}\vert_A=0$ we have $\ind{A}\overline{h}=0$ and $\overline{h}=(\overline{h}\resAc)\upX$, whence
	\begin{equation*}
		\bigl(Q\,\overline{h}\bigr)\resAc = 	\bigl(Q\,(\overline{h}\resAc)\upX\bigr)\resAc = G\,\overline{h}\resAc\,.
	\end{equation*}
	Noting that $\ind{A^c}\resAc=\ones$, after multiplying with $\Delta$ we find that $\Delta\ones + \Delta G\,\overline{h}\resAc=0$. Comparing with the definition of $H$, we have
	\begin{align*}
		H(\overline{h}\resAc) = \overline{h}\resAc + \Delta\ones + \Delta G\,\overline{h}\resAc = \overline{h}\resAc + 0 = \overline{h}\resAc\,,
	\end{align*}
	so $\overline{h}\resAc$ is indeed a fixed-point of $H$. Since we already established that $H$ has a unique fixed-point, we conclude from Equation~\eqref{eq:thm:lower_upper_hitting_system:abstract_banach_limit} that
	\begin{equation}\label{eq:thm:lower_upper_hitting_system:concrete_banach_limit}
		\overline{h}\resAc = \lim_{n\to+\infty} H^n(g)\quad\text{for all $g\in\gamblesAc$.}
	\end{equation}
	Next let us show that $H$ is monotone. To this end, fix any $f,g\in\gamblesAc$ such that $f\leq g$; then clearly also $f\upX \leq g\upX$. Since $\Delta>0$ is such that $\Delta\norm{Q}\leq 1$, it follows from~\cite[Prop 4.9]{krak2021phd} that $(I+\Delta Q)$ is a transition matrix. By the monotonicity of transition matrices~\cite[Prop 3.32]{krak2021phd}, we find that therefore
	\begin{equation*}
		(I+\Delta Q)f\upX \leq (I+\Delta Q)g\upX\,,
	\end{equation*}
	which in turn implies that
	\begin{align*}
		(I+\Delta G)f &= \bigl((I+\Delta Q)f\upX\bigr)\resAc \\
		&\leq \bigl((I+\Delta Q)g\upX\bigr)\resAc = (I+\Delta G)g\,.
	\end{align*}
	From the definition of $H$ we therefore conclude that $H(f)\leq H(g)$. Since $f,g$ with $f\leq g$ are arbitrary, this concludes the proof of the monotonicity of $H$.

	Now, let us define $\overline{H}:\gamblesX\to\gamblesX$ for all $f\in\gamblesX$ as
	\begin{equation*}
		\overline{H}(f) \coloneqq \Delta\ind{A^c} + \ind{A^c}(I+\Delta\urate)f\,.
	\end{equation*}
	We first note that, since $\Delta\norm{\lrate}\leq 1$, it follows from~\cite[Prop 5]{de2017limit} that $(I+\Delta\lrate)$ is a lower transition operator. From the conjugacy of $\lrate$ and $\urate$, we have for any $f\in\gamblesX$ that
	\begin{align*}
		-(I+\Delta\lrate)(-f) &= f + \Delta\cdot-\lrate(-f) \\
		&= f + \Delta\urate f = (I+\Delta\urate)f\,,
	\end{align*}
	which implies that $(I+\Delta\urate)$ is the upper transition operator that is conjugate to the lower transition operator $(I+\Delta\lrate)$. By the monotonicity of upper transition operators---see Section~\ref{subsec:imprecise_dynamics}---this implies that $\overline{H}$ is monotone, or in other words that for all $f,g\in\gamblesX$ with $f\leq g$ it holds that $\overline{H}(f)\leq\overline{H}(g)$.
	
	Let us next show that
	\begin{equation}\label{eq:thm:lower_upper_hitting_system:bound_different_iters}
		\overline{H}(f) \geq H(f\resAc)\upX\quad \text{for all $f\in\gamblesX$ with $f\vert_A=0$.}
	\end{equation}
	Indeed, if $f\vert_A=0$ then $(f\resAc)\upX = f$, and since $Q\in\rateset$, it follows from the definitions of $\urate$ and $G$ that then
	\begin{equation*}
		(\urate\,f)\resAc \geq (Q\,f)\resAc = (Q\,(f\resAc)\upX)\resAc = Gf\resAc\,.
	\end{equation*}
	Hence we have
	\begin{align*}
		\overline{H}(f)\resAc &= \Delta\ones + f\resAc + \Delta(\urate\,f)\resAc \\
		&\geq \Delta\ones + f\resAc + \Delta Gf\resAc = H(f\resAc)\,.
	\end{align*}
	Moreover, we immediately have from the definition that $\overline{H}(f)\vert_A = 0 = \bigl(H(f\resAc)\upX\bigr)\vert_A$, and so Equation~\eqref{eq:thm:lower_upper_hitting_system:bound_different_iters} indeed holds.
	
	Next, we note that for any $f\in\gamblesX$ it holds that $\overline{H}(f)\vert_A=0$, and so by Equation~\eqref{eq:thm:lower_upper_hitting_system:bound_different_iters} we find that
	\begin{equation*}
		\overline{H}(\overline{H}(f)) \geq H(\overline{H}(f)\resAc)\upX\,.
	\end{equation*}
	Provided that also $f\vert_A=0$, then using the previously established monotonicity of $H$ we obtain
	\begin{align*}
		\overline{H}(\overline{H}(f)) &\geq H(\overline{H}(f)\resAc)\upX \\
		&\geq H(H(f\resAc)\upX\resAc)\upX\,.
	\end{align*}
	We clearly have $H(f\resAc)\upX\resAc = H(f\resAc)$, whence
	\begin{equation*}
		\overline{H}^2(f) = \overline{H}(\overline{H}(f)) \geq H(H(f\resAc))\upX = (H^2(f\resAc))\upX\,.
	\end{equation*}
	Indeed, we can repeat this reasoning for $n\in\nats$ steps, to conclude that
	\begin{equation}\label{eq:thm:lower_upper_hitting_system:iterative_bound}
		\overline{H}^n(f) \geq (H^n(f\resAc))\upX\quad\text{for all $f\in\gamblesX$ with $f\vert_A=0$.}
	\end{equation}
	
	Now suppose \emph{ex absurdo} that there is some non-negative $g\in\gamblesX$, such that $g(x)<\overline{h}(x)$ for some $x\in\states$, and such that $\ind{A}g=\ind{A^c}+\ind{A^c} \urate\,g$. Since $g$ is non-negative and $\overline{h}\vert_A=0$, we must have that $x\in A^c$. Moreover, we clearly have $\ind{A}g=0$, which implies that $g\vert_A=0$ and so $g=\ind{A^c}g$. Hence it follows that $\Delta\ind{A^c} + \ind{A^c}\Delta\urate g = 0$, and we find that $\overline{H}(g) = \ind{A^c}g=g$.
	Hence $g$ is a fixed point of $\overline{H}$. Since $g\vert_A=0$, and using Equation~\eqref{eq:thm:lower_upper_hitting_system:iterative_bound}, this implies that for any $n\in\natswith$ we have
	\begin{equation*}
		g = \overline{H}^n(g) \geq (H^n(g\resAc))\upX\,.
	\end{equation*}
	Recalling that $x\in A^c$ is such that $g(x) < \overline{h}(x)$, we use Equation~\eqref{eq:thm:lower_upper_hitting_system:concrete_banach_limit} to take limits in $n$ and find that
	\begin{equation*}
		g(x) \geq \lim_{n\to+\infty} H^n(g\resAc)(x) = \overline{h}(x) > g(x)\,,
	\end{equation*}
	which is a contradiction.
\end{proof}

\end{document}